\documentclass[12pt,reqno]{amsart}
\usepackage{fullpage}
\usepackage{times}
\usepackage{amsfonts, amssymb, amsmath, amsthm, mathtools, thmtools}
\usepackage[utf8]{inputenc}
\usepackage[english]{babel}
\usepackage{bbm}
\usepackage{enumerate}
\usepackage{bm}
\usepackage{graphicx}
\usepackage{mathrsfs}
\usepackage[colorlinks=true, pdfstartview=FitH, linkcolor=blue, citecolor=blue, urlcolor=blue]{hyperref}
\usepackage{tikz-cd}
\usepackage{tabstackengine}
\stackMath
\usepackage{comment}
\usepackage[capitalise]{cleveref}
\numberwithin{equation}{section}

\newtheorem{thm}{Theorem}[section]
\newtheorem{lem}[thm]{Lemma}
\newtheorem{prop}[thm]{Proposition}
\newtheorem{cor}[thm]{Corollary}
\newtheorem{conj}[thm]{Conjecture}
\newtheorem{externaltheorem}{Theorem}

\theoremstyle{definition}

\newtheorem{question}[thm]{Question}

\newtheorem{rem}[thm]{Remark}

\newcommand{\N}{\mathbb{N}}
\newcommand{\Z}{\mathbb{Z}}
\newcommand{\F}{\mathbb{F}}

\title{Diophantine tuples and product sets in shifted powers}
\author{Ernie Croot}
\address{School of Mathematics\\ Georgia Institute of Technology\\ Atlanta, GA 30332\\ United States}
\email{ernest.croot@math.gatech.edu}
\author{Chi Hoi Yip}
\address{School of Mathematics\\ Georgia Institute of Technology\\ Atlanta, GA 30332\\ United States}
\email{cyip30@gatech.edu}
\begin{document}

\begin{abstract}
Let $k\geq 2$ and $n\neq 0$. A Diophantine tuple with property $D_k(n)$ is a set of positive integers $A$ such that $ab+n$ is a  $k$-th power for all $a,b\in A$ with $a\neq b$. Such generalizations of classical Diophantine tuples have been studied extensively. In this paper, we prove several results related to robust versions of such Diophantine tuples and discuss their applications to product sets contained in a nontrivial shift of the set of all perfect powers or some of its special subsets. In particular, we substantially improve several results by B\'{e}rczes--Dujella--Hajdu--Luca, and Yip. We also prove several interesting conditional results. Our proofs are based on a novel combination of ideas from sieve methods, Diophantine approximation, and extremal graph theory.
\end{abstract} 

\maketitle

\section{Introduction}
A set $\{a_1,a_2,\ldots,a_m\}$ of distinct positive integers is called a
\emph{Diophantine $m$-tuple} if every pair of distinct elements satisfies
\[
a_i a_j + 1 \text{ is a perfect square} \qquad (1\le i<j\le m).
\]
In other words, multiplying any two different elements of the set produces a
number that lies just one below a perfect square. Here, the shift is essential:
without it the problem becomes trivial, since taking $a_i=t b_i^2$ for a fixed
$t$ forces every product $a_i a_j$ to be a square.

The study of Diophantine tuples goes back to Fermat: the classical set $\{1,3,8,120\}$ is the first known example of a Diophantine $4$-tuple. More broadly, Diophantine tuples and their generalizations and variants serve as a prototypical testing ground for how strong arithmetic constraints can govern the size and structure of sets of integers. Questions about their existence, classification, and maximal size naturally lead to challenging Diophantine equations and finiteness problems, and they have motivated a range of techniques that extend well beyond the original setting.

Despite their elementary formulation, Diophantine tuples have rich connections with Diophantine equations and Diophantine approximation, and they are closely related to several central themes in number theory, including elliptic curves, arithmetic geometry, sieve methods, and arithmetic combinatorics. We refer to the recent book of Dujella \cite{D24} for background and a comprehensive overview.

In this paper, we obtain improved upper bounds---both unconditional and
conditional---on the size of several well-studied families of variants of
Diophantine tuples, and we clarify how these variants are closely related.

Let $n$ be a nonzero integer and let $k \ge 2$. A set $A$ of distinct positive integers is a \textit{Diophantine tuple with property $D_{k}(n)$} if the product of any two distinct elements in $A$ is $n$ less than a perfect $k$-th power. Following the standard notations, we also write
\[M_{k}(n)=\sup \{|A| \colon A\subseteq{\mathbb{N}} \text{ satisfies the property }D_{k}(n)\}.\] 
These natural generalized notions of Diophantine tuples are of particular interest; see, for example \cite{BDHL11, BD03, DKM22, D02, DL05, GM20, G01, KYY, Y24, Y24+}. The best-known upper bound on $M_k(n)$ is of the form $M_k(n)\ll_k \log |n|$; see \cite{KYY, Y24, Y24+} for the best-known implied constant depending on $k$. Here we used the Vinogradov notation $\ll$: we write $X \ll Y$ if there is an absolute constant $C>0$ so that $|X| \leq CY$. Under the Uniformity Conjecture \cite{CHM} (a consequence of the Bombieri–Lang conjecture), it is well-known that for each $k \geq 2$, there is a constant $C_k$ such that $M_k(n)\leq C_k$ holds for all nonzero integers $n$ \cite{D02,KYY}; see also Remark~\ref{rem:uniformity}. In Corollary~\ref{cor:absolute}, we show that assuming both the Uniformity Conjecture and the Lander--Parkin--Selfridge conjecture~\cite{LPS67}, there is a constant $C$ such that $M_k(n)\leq C$ holds for all $k\geq 2$ and $n \neq 0$.

In 2002, Gyarmati, S\'ark\"ozy, and Stewart~\cite{GSS02} initiated the study of two closely related variants of Diophantine tuples by enlarging the set of perfect $k$-th powers (with a fixed $k$) to the set of perfect $k$-th powers with bounded $k$ or the set of all perfect powers. More precisely, set $$V_d=\{x^k: x \in \N, 2\leq k \leq d\}$$ for each integer $d\geq 2$; then the set of perfect powers is 
$$V_{\infty}:=\bigcup_{d\geq 2}V_d=\{x^k: x,k \in \N, k \geq 2\}.$$ They studied the size of a set of positive integers $A$ such that $aa'+1$ is in $V_d$ (where $2\leq d\leq \infty$) for all $a, a' \in A$ that are distinct. More generally, one can study the same question by replacing the shift $1$ with any nonzero shift; see, for example, \cite{BDHL11, BDHT16}. For brevity, we extend the familiar notions $D_k(n)$ and $M_k(n)$ to this more general setting. Let $n$ be a nonzero integer and let $2\leq d\leq \infty$, we say a set $A$ of positive integers is a \textit{Diophantine tuple with property $D_{\leq d}(n)$} if the product of any two distinct elements in $A$ is $n$ less than an element in $V_d$, and we denote 
\[M_{\leq d}(n)=\sup \{|A| \colon A\subseteq{\mathbb{N}} \text{ satisfies the property }D_{\leq d}(n)\}.\] 
In particular, the case $n=1$ is well-studied. Gyarmati, S\'ark\"ozy, and Stewart~\cite{GSS02} proved that if $2\leq d<\infty$ and $A \subseteq \{1,2,\ldots, N\}$ is a Diophantine tuple with property $D_{\leq d}(1)$, then $|A|\ll \frac{d^2}{(\log d)^2}\log \log N$. They then deduced that if $A \subseteq \{1,2,\ldots, N\}$ is a Diophantine tuple with property $D_{\leq \infty}(1)$, then $|A|\ll (\log N)^2/\log \log N$. These two results have been improved by various authors~\cite{BG04, DEGS05, L05, GS07, S08}. Regarding their first result, the best-known improvement is due to Bugeaud and Gyarmati \cite{BG04}, where they showed that $M_{\leq d}(1)\ll (d/\log d)^2$ for $2\leq d<\infty$. As for the second result, the best-known bound is $|A|\ll (\log N)^{2/3} (\log \log N)^{1/3}$, due to Stewart \cite{S08}. 

As a highly relevant quantity, B\'{e}rczes, Dujella, Hajdu, and Luca~\cite{BDHL11} introduced the following function: for $x\geq 1$, let $f(x)$ be the maximal $K$ such that there exists a set $A\subseteq [1,x]\cap \N$ with $K$ elements and some $1\leq n \leq x$ such that $ab+n$ is a perfect power for all $a,b \in A$ with $a\neq b$. Equivalently, 
$$f(x)=\max \{|A|: A\subseteq [1,x]\cap \N \text{ has property } D_{\leq \infty}(n) \text{ for some } 1\leq n \leq x\}.$$
For our purpose, it is also natural to define a similar function to include those negative $n$'s:
$$\widetilde{f}(x)=\max \{|A|: A\subseteq [1,x]\cap \N \text{ has property } D_{\leq \infty}(n) \text{ for some } 1\leq |n| \leq x\}.$$
In \cite[Theorem 2 and Remark 2]{BDHL11}, B\'{e}rczes, Dujella, Hajdu, and Luca showed that $f(x)\leq x^{2/3+o(1)}$ as $x\to \infty$, and $f(x)\geq \lfloor(\log \log x/ 2\log \log \log x)^{1/3} \rfloor$ for $x>e^{e^e}$.

Our new results are presented in Section~\ref{sec:newresults}. In Section~\ref{subsec:main}, we state our two main theorems. The first extends the upper bound for $M_{\le d}(1)$ due to Bugeaud and Gyarmati \cite{BG04} to $M_{\le d}(n)$ for an arbitrary nonzero integer $n$. The second establishes that
$$
f(x)\leq \exp(L(\log \log x)^2)
$$
for some absolute constant $L$, which is a substantial improvement over the bound $x^{2/3+o(1)}$ of B\'{e}rczes, Dujella, Hajdu, and Luca \cite{BDHL11}. The proofs combine tools from sieve methods, Diophantine approximation, and extremal graph theory.

In Section~\ref{subsec:bipartite}, we provide further background and record several additional new results needed for the proofs. Shifted product sets contained in perfect powers are significantly more difficult to handle than those contained in perfect $k$-th powers for a fixed $k$, so it is natural to separate the contributions coming from squares, cubes, fourth powers, and so on. Previous approaches then used Ramsey theory to recombine these contributions, via bounds for $k$-th power Diophantine tuples. The key novelty of our method is to introduce and study robust variants of Diophantine tuples that align with this separation. Using tools from extremal graph theory, we reduce the problem to the study of bipartite variants of Diophantine tuples introduced by the second author \cite{Y24}. We then study these bipartite Diophantine tuples primarily using sieve methods, assisted by an analysis of their finite-field models via character sum estimates. This framework controls the contributions of different $k$-th powers much more efficiently, and consequently yields the improved bound on $f(x)$. Finally, in Section~\ref{subsec:cond}, we discuss several related conditional results; for instance, assuming the ABC conjecture, we prove that $M_{\leq \infty} (n)\leq f(2|n|^{17})$.  

The structure of the rest of the paper is the following. In Section~\ref{sec:prelim}, we provide additional background and prove some preliminary results. In Section~\ref{sec:large}, we bound the contribution of large elements in a Diophantine tuple with the desired property. In Section~\ref{sec:finitefield}, we study various finite field analogues of Diophantine tuples as a preparation to apply sieve methods. In Section~\ref{sec:bipartite}, we apply sieve methods to prove the improved bounds on bipartite Diophantine tuples stated in Section~\ref{subsec:bipartite}. Then, in Section~\ref{sec:main}, we combine results from all previous sections to prove Theorem~\ref{thm:Vd} and Theorem~\ref{thm:infty}. Finally, in Section~\ref{sec:cond}, we prove the conditional results stated in Section~\ref{subsec:cond}. 

Throughout the paper,~$p$ always denotes a prime, and $\sum_p$ and $\prod_p$ represent sums and products over all primes.

\section{New results}\label{sec:newresults}
\subsection{Main results}\label{subsec:main}

Our first result extends the result of Bugeaud and Gyarmati \cite{BG04} to $M_{\leq d}(n)$ for a general nonzero integer $n$. 

\begin{thm}\label{thm:Vd}
Let $d,n$ be integers with $2\leq d<\infty$ and $n\neq 0$. Then we have
$$
M_{\leq d}(n)\ll \frac{d^2}{(\log d)^2}+e^{dL}\log |n|,
$$
where $L$ is an absolute constant, and the implied constant is absolute. In particular, if $2\leq d<\infty$ is fixed, then $M_{\leq d}(n)\ll \log |n|$. 
\end{thm}

The constant $L$ in the above theorem (and a few other results in this paper) is related to Linnik's constant. It naturally comes from a quantitative version of Linnik's theorem used in our proof. Also note that if $d<\infty$ is fixed, we have $M_{\leq d}(n)\ll_{d} \log |n|$, which is of the same shape as the best-known upper bound on $M_d(n)$.

Our second result concerns the case $d=\infty$.

\begin{thm}\label{thm:infty}
Let $n$ be a nonzero integer and $N$ be a positive integer. If $A \subseteq [1,N]$ is a Diophantine tuple with property $D_{\leq \infty}(n)$, then there is an absolute constant $L$ such that
$$
|A|\ll \exp\big(L(\log \log |n|)^2\big)+ \log N,
$$    
where the implied constant is absolute. In particular, if $L (\log \log |n|)^2\leq \log \log N$, then $|A|\ll \log N$.
\end{thm}

Next, we consider the function $f(x)$. We provide a significant improvement on the $x^{2/3+o(1)}$ upper bound by B\'{e}rczes, Dujella, Hajdu, and Luca \cite{BDHL11} in the following corollary. The corollary follows from~\cref{thm:infty} immediately. 

\begin{cor}\label{cor:f}
There is an absolute constant $L$ such that
$$
f(x)\leq \widetilde{f}(x)\ll \exp\big(L(\log \log x)^2\big).
$$    
\end{cor}

In the next subsection, we provide an overview of our proof of Theorem~\ref{thm:Vd} and Theorem~\ref{thm:infty}, give some additional backgrounds, and state our additional contributions.

\subsection{Improved bounds on bipartite Diophantine tuples}\label{subsec:bipartite}
Tools from graph theory have been applied frequently to connect the property $D_k(n)$ and the property $D_{\leq d}(n)$. To see a quick connection, recall that the \emph{(multi-colored) Ramsey number} $R(m_1, m_2, \ldots, m_s)$ is the smallest positive integer $R$ such that any coloring of the edges of the complete graph on $R$ vertices with $s$ colors results in the existence of an $1\leq i \leq s$, such that there exists a complete monochromatic subgraph in the $i$-th color with $m_i$ vertices. Let $2\leq d<\infty$, and let $p_1<p_2<\ldots<p_{\ell}$ be all the primes at most $d$. If $n$ is a nonzero integer, then we have 
\begin{equation}\label{eq:Ramsey}
M_{\leq d}(n)\leq R(M_{p_1}(n), M_{p_2}(n), \ldots, M_{p_\ell}(n)).    
\end{equation}
Indeed, if $A$ is a Diophantine tuple with property $M_{\leq d}(n)$, then we can build a complete graph with vertex set $A$. For each edge $ab$ in the graph, we color it with color $i$ if $i$ is the smallest number such that $ab+n$ is a perfect $p_i$-th power. This simple observation immediately shows that $M_{\leq d}(n)$ is finite since we know each $M_{p_i}(n)\ll_{p_i} \log |n|$, except that the upper bound obtained from inequality~\eqref{eq:Ramsey} is potentially very large and far from the truth. 

Our proofs of Theorem~\ref{thm:Vd} and Theorem~\ref{thm:infty} are based on a novel combination of ideas from sieve methods, Diophantine approximation, and extremal graph theory. In particular, to prove these two theorems, it is natural to study the following question related to a ``robust version" of Diophantine tuples, which is also of independent interest.

\begin{question}\label{q}
Let $\delta \in (0,1]$ be a real number. Let $k\geq 2$ and let $n$ be a nonzero integer. Suppose that $X$ is a subset of positive integers and among all subsets $\{a,b\}$ of $X$ with size $2$, there are at least $\delta \binom{|X|}{2}$ pairs of them such that $ab+n$ is a perfect $k$-th power. Can we give an upper bound on $|X|$?   
\end{question}

Note that when $\delta=1$, a set $A$ satisfying the condition in the question above is a Diophantine tuple with property $D_k(n)$. Also note that if $N$ is a sufficiently large integer and $A \subseteq [1,N]$ is a Diophantine tuple with property $D_{\leq \infty}(n)$ for some $1\leq |n|\leq N$, then there is a prime $k$ with $2\leq k \leq \log N$, such that $A$ satisfies the assumption in Question~\ref{q} with $\delta=\frac{1}{\log N}$. 

To answer Question~\ref{q}, we take advantage of tools from extremal graph theory; see Section~\ref{subsec:exgraph}. In particular, in view of the K\"ov\'ari--S\'os--Tur\'an theorem (Lemma~\ref{lem:KST}), it suffices to analyze the bipartite variant of Diophantine tuples; see Remark~\ref{rem:q}. 
Bipartite Diophantine tuples were recently introduced by the second author \cite{Y24, Y26} in connection with the work of Hajdu and S\'{a}rk\"{o}zy on the multiplicative decompositions of a small perturbation of the set of shifted $k$-th powers~\cite{HS18, HS18b, HS20}. However, the same objects have been studied, for example by Gyarmati \cite{G01}, Bugeaud and Dujella \cite{BD03}, and Bugeaud and Gyarmati \cite{BG04}, two decades ago. More precisely, following \cite{Y24}, for each $k \ge 2$ and each nonzero integer $n$, we call a pair of sets $(A, B)$ a \textit{bipartite Diophantine tuple with property $BD_{k}(n)$} if $A, B$ are two subsets of $\N$ with size at least $2$, such that $ab+n$ is a $k$-th power for each $a \in A$ and $b \in B$. From the viewpoint of number theory, such bipartite variants of Diophantine tuples are also natural and useful local objects to study, since forbidden local structure often helps us to understand the global structure. 

Generally speaking, bipartite Diophantine tuples are much harder to study compared to Diophantine tuples, since being a Diophantine tuple imposes many more restrictions. For example, the quantities $M_2(1)$ and $M_2(-1)$ were well-studied \cite{D04, DF05} and eventually He, Togb\'e, and Ziegler \cite{HTZ19} proved that $M_2(1)=4$, and Bonciocat, Cipu, and Mignotte \cite{BCM22} proved that $M_2(-1)=3$. Nevertheless, it remains an open question to show that if $(A, B)$ is a bipartite Diophantine tuple with property $BD_{2}(1)$ (or $BD_{2}(-1)$, resp.), then $\min \{|A|,|B|\}$ is bounded by an absolute constant \cite{BHP25, BG04, Y26}; and the more general question for $BD_2(n)$ with $|n|\geq 2$ appears to be even harder. The second author \cite{Y24} studied the same question for $k\geq 3$. In particular, it was shown in \cite[Theorem 2.2]{Y24} that if $k\geq 3$ is fixed, then $\min \{|A|,|B|\}\ll \log |n|$ holds for all bipartite Diophantine tuples $(A, B)$ with property $BD_{k}(n)$. In the same paper, the following stronger result was proved. To state it, following \cite{Y24}, we define the following constants:
\begin{align*}
&r_3=9, \quad  r_4=6, \quad r_5=5, \quad \text{ and } \quad r_k=4 \quad \text{ for } k \geq 6;\\
&s_3=6, \quad  s_4=4, \quad s_5=3, \quad \text{ and } \quad s_k=2 \quad \text{ for } k \geq 6;
\end{align*}
$$t_3=\frac{15399}{938}, \quad t_4=\frac{34}{3}, \quad t_5=\frac{97}{23}, \quad t_6=\frac{29}{4}, \text{ and } \quad t_k=\frac{k^2+k-4}{k^2-6k+6} \quad \text{ for } k \geq 7.
$$
The above constants will appear at many places in the paper, and we will use them without any further mentioning.
\begin{externaltheorem}[{\cite[Theorem 2.3]{Y24}}]\label{thm:old}
Let $k\geq 3$ be fixed, and $n$ be a nonzero integer. Let $(A,B)$ be a bipartite Diophantine tuple with property $BD_{k}(n)$. Then we have:
\begin{enumerate}
    \item If $|A|\geq r_k$, then $|B|\leq |n|^{\frac{t_k}{k}+o(1)}$ as $|n|\to \infty$.
    \item If $|A|\geq \frac{\log \log |n|+3.3}{\log (k-1)}+8$, then $|B|\leq |n|^{\frac{1}{k-2}+o(1)}$ as $|n|\to \infty$.
\end{enumerate}
\end{externaltheorem}

The proof of \cref{thm:old} is based on a combination of tools from Diophantine approximation and sieve methods, and it relied on a result of Bourgain and Demeter \cite{BD18} on the number of $k$-th powers inside arithmetic progressions (see \cite[Proposition 4.6]{Y24}). We provide a substantial improvement on Theorem~\ref{thm:old} in the next two results. In our proof, we use a more sophisticated application of sieve methods (building on the refined finite field models in Section~\ref{sec:finitefield}) instead of considering $k$-th powers in arithmetic progressions. 

For positive integers $k$ and $m$, we define the following constant:
\begin{equation}\label{eq:theta}
\theta_{k,m}:=\sum_{\substack{1\leq i \leq k\\\gcd(i,k)=1}}\gcd(i-1,k)^m.    
\end{equation}

\begin{thm}\label{thm:thm1}
Let $k \geq 3$ be fixed. Let $n$ be a nonzero integer and  $(A,B)$ be a bipartite Diophantine tuple with property $BD_{k}(n)$. Then as $|n|\to \infty$, the following estimate holds uniformly over an integer $m \geq r_k-s_k$ and $m=o(\log \log |n|)$:
$$
|A|\geq m+s_k \implies |B|\leq |n|^{\frac{t_k\phi(k)}{\theta_{k,m}}+o(1)}.
$$
\end{thm}

Let $(A,B)$ be a bipartite Diophantine tuple with property $BD_{k}(n)$, where $k\geq 3$ and $n\neq 0$. Let us compare Theorem~\ref{thm:old}(1) and Theorem~\ref{thm:thm1}.
Since $\gcd(i-1,k)=1\geq 1$ for all $i$ and $\gcd(0,k)=k$, we have
$\theta_{k,m}\geq k^m+\phi(k)-1.$
Thus, if $|A|\geq r_k$, by setting $m=r_k-s_k$, Theorem~\ref{thm:thm1} implies that 
\begin{equation}\label{eq:con1}    
|B|\leq |n|^{\frac{t_k \phi(k)}{k^2+\phi(k)-1}+o(1)},
\end{equation}
always improving Theorem~\ref{thm:old}(1). Indeed, if $k$ is a prime, then inequality~\eqref{eq:con1} becomes
$$
|B|\leq |n|^{\frac{t_k (k-1)}{k^2+k-2}+o(1)}=|n|^{\frac{t_k}{k+2}+o(1)};
$$
when $k$ is composite, we get an even better exponent. When $|A|\gg \log \log |n|$, Theorem~\ref{thm:thm1} shows that $|B|=|n|^{o(1)}$, which improves Theorem~\ref{thm:old}(2). Our next result provides further refinement in this setting. 

\begin{thm}\label{thm:loglogN2}
Let $k\geq 3$ be fixed, and $n$ be a nonzero integer with $|n|$ sufficiently large. Assume that $(A,B)$ is a bipartite Diophantine tuple with property $BD_{k}(n)$. Then there are two constants $C_1,C_2$ only depending on $k$, such that 
\[\text{either} \quad |A|\leq C_1 \log \log |n|, \qquad \text{or} \quad |B|\leq C_2 \log |n| \cdot (\log \log |n|)^2.
\]
\end{thm}

As pointed out by a referee, Theorem~\ref{thm:loglogN2} can be viewed as a multiplicative analogue of a recent result of Elsholtz and Wurzinger \cite[Theorem 1.2]{EW24} on sumsets contained in the set of $k$-th powers.

For our applications to Theorem~\ref{thm:infty}, the following uniform version of Theorem~\ref{thm:loglogN2} will be a key ingredient.

\begin{thm}\label{thm:loglogN}
Let $k\geq 2$ and $n,N$ be integers such that $1\leq |n|\leq N$ and $k\leq 5\log N$. Assume that $A,B \subseteq \{1,2,\ldots, N\}$ such that $(A,B)$ is a bipartite Diophantine tuple with property $BD_{k}(n)$. Then there are two absolute constants $L_1,L_2$ such that either $|A|\leq L_1 \log \log N$ or $|B|\leq (\log N)^{L_2}$. 
\end{thm}

\subsection{Conditional results}\label{subsec:cond}
In this section, we state our new conditional upper bounds on various notions of Diophantine tuples discussed earlier. 

There are many conjectures related to sums of perfect powers, and they are often useful in Diophantine questions. The following conjecture, due to Lander, Parkin, and Selfridge \cite{LPS67}, is widely believed.

\begin{conj}[Lander--Parkin--Selfridge conjecture]\label{conj:LPS}
Let $m,n,k$ be positive integers. If $\sum _{i=1}^{n}a_{i}^{k}=\sum _{j=1}^{m}b_{j}^{k},$ where $a_1,a_2, \ldots, a_n, b_1, b_2, \ldots, b_m$ are positive integers such that $a_i \neq b_j$ for all $1\leq i \leq n$ and $1\leq j \leq m$, then $m+n \geq k$.
\end{conj}

For our purpose, we need the following special case of Conjecture~\ref{conj:LPS}.
\begin{conj}[Lander--Parkin--Selfridge conjecture, special case]\label{conj:LPS2}
If $k\geq 25$ is an integer, then there do not exist distinct positive integers $a_1,b_1, \ldots, a_{12}, b_{12}$ such that $\sum_{i=1}^{12}a_{i}^{k}=\sum _{i=1}^{12}b_{i}^{k}.$   
\end{conj}

We remark that the results listed in this section still hold, at the cost of larger implied constants, if Conjecture~\ref{conj:LPS2} holds under the weaker assumption that $k\geq k_0$ for some fixed $k_0$.

Our next result assumes Conjecture~\ref{conj:LPS2}.

\begin{thm}\label{thm:absolute}
Assume Conjecture~\ref{conj:LPS2}. If $k \geq 25$, $n \neq 0$, and $(A,B)$ is a bipartite Diophantine tuple with property $BD_{k}(n)$, then $\max\{|A|,|B|\}\leq 21736$. In particular, $M_k(n)\leq 21738$ for all $k \geq 25$ and $n \neq 0$. 
\end{thm}

We did not attempt to optimize the constant in the above theorem. In the next two remarks, we explain some motivations behind Theorem~\ref{thm:absolute}.

\begin{rem}
Theorem~\ref{thm:absolute} partially addresses a question asked in \cite[Remark 4.9]{Y24}, where it is asked whether the ABC conjecture could be used to show that if $k\geq 3$ and $n$ is a nonzero integer, then $|A||B|$ is absolutely bounded among all bipartite Diophantine tuples $(A,B)$ with property $BD_{k}(n)$  \footnote{Note that when $k=2$, because of Pell's equations, for example when $A=\{1,2\}$, there exists an infinite set $B$ such that $AB+1$ is contained in the set of perfect squares. In particular, Theorem~\ref{thm:absolute} does not hold for $k=2$.}. Theorem~\ref{thm:absolute} answers this question in the affirmative in a much stronger form, provided that $k \geq 25$, assuming Conjecture~\ref{conj:LPS2}. While Conjecture~\ref{conj:LPS2} does not follow from the ABC conjecture, a slightly weaker version of the conjecture follows immediately from the $n$-conjecture \cite{BB94}, which is a generalization of the ABC conjecture to more variables.
\end{rem}

\begin{rem}\label{rem:uniformity}
The Uniformity Conjecture, due to Caporaso, Harris, and Mazur~\cite{CHM}, states that if $g \geq 2$, then there is a constant $C_g$, such that the number of $\mathbb{Q}$-points of each curve defined over the rationals $\mathbb{Q}$ of genus $g$ is bounded by $C_g$. Note that for $i \in \{1,2\}$, the hyperelliptic curve $$y^2=(x-1)(x-2)\cdots (x-(2g+i))$$ has genus $g$ and has at least $2g+i$ integral points. In particular, $C_g \to \infty$ as $g \to \infty$. 

If $A$ is a Diophantine tuple with property $D_k(n)$ with $k\geq 4$ and $n$ nonzero, we can take $3$ distinct elements $a_1,a_2,a_3$ from $A$ and consider the superelliptic curve
$$
E: y^k=(a_1x+n)(a_2x+n)(a_3x+n).
$$
Since each $a\in A\setminus \{a_1,a_2,a_3\}$ corresponds to an integral point of $E$, we have $|A|\leq |E(\Z)|+3\leq |E(\mathbb{Q})|+3 \leq C_g+3$ assuming the Uniformity Conjecture, where $g$ is the genus of $E$. From the Riemann-Hurwitz formula, $g=k-2$ if $3\mid k$ and $g=k-1$ if $3\nmid k$. Thus, the Uniformity Conjecture only predicts that $M_k(n)\leq \widetilde{C_k}$ for all $n \neq 0$, where $\widetilde{C_k} \to \infty$ as $k \to \infty$. This is much weaker compared to the prediction from Theorem~\ref{thm:absolute}.
\end{rem}

If $n$ is a fixed nonzero integer, then as $k$ increases, it is likely that $M_k(n)$ decreases since being a perfect $k$-th power for a larger $k$ is more restrictive. Indeed, it is proved in \cite[Theorem 2.4 and Corollary 2.6]{Y24} that $M_k(n)\leq 9$ for all $k\geq 2$ provided that $|n|=1$, and $M_k(n)\leq 19$ if $|n|\geq 2$ and $k \geq 2\log |n|+2$. Thus, it is plausible that there is an absolute constant $C$ such that $M_k(n)\leq C$ for all $k\geq 2$ and $n\neq 0$. In view of \cref{rem:uniformity}, by applying the Uniformity Conjecture for $k\leq 24$, and Theorem~\ref{thm:absolute} for $k\geq 25$, this heuristic follows immediately. 

\begin{cor}\label{cor:absolute}
Assume the Uniformity Conjecture and Conjecture~\ref{conj:LPS2}. There is an absolute constant $C$ such that $M_k(n)\leq C$ for all $k\geq 2$ and $n\neq 0$.    
\end{cor}

Next, we consider conditional upper bounds on $M_{\leq d}(n)$ and $\widetilde{f}(x)$.

\begin{thm}\label{thm:absolutedf}
Assume the Uniformity Conjecture and Conjecture~\ref{conj:LPS2}. There is an absolute constant $C'$ such that $M_{\leq d}(n)\leq C' (d/\log d)^2$ for all $2\leq d<\infty$ and $n\neq 0$. In particular, $\widetilde{f}(x)\ll (\log x/\log \log x)^2$.      
\end{thm}

We also prove a weaker bound only assuming Conjecture~\ref{conj:LPS2}.

\begin{thm}\label{thm:LPS}
Assume Conjecture~\ref{conj:LPS2}. Then $\widetilde{f}(x)\ll (\log x)^{4}$.
\end{thm}

Finally, we turn our attention to conditional upper bounds on $M_{\leq \infty}(n)$. So far there is no unconditional proof that $M_{\leq \infty}(1)<\infty$, let alone $M_{\leq \infty}(n)$ for a general nonzero integer $n$. We first recall the ABC conjecture.

\begin{conj}[ABC conjecture]
For each $\epsilon>0$, there is a constant $C_\epsilon$, such that whenever $a,b,c$ are nonzero integers with $\gcd(a,b,c)=1$ and $a+b=c$, we have 
$$
\max \{|a|,|b|,|c|\} \leq C_{\epsilon} \operatorname{rad} (abc)^{1+\epsilon}.
$$    
\end{conj}
Under the ABC conjecture, Luca \cite{L05} showed that $M_{\leq \infty}(1)$ is finite. More generally, B\'{e}rczes, Dujella, Hajdu, and Luca \cite[Theorem 4]{BDHL11} showed that under the ABC conjecture, $M_{\leq \infty}(n)$ is finite whenever $n$ is a nonzero integer. More precisely, their proof shows that $M_{\leq \infty}(n)\leq c_0|n|^3+ R(n)$, where $c_0$ is an absolute constant and
$$
R(n)=R(C_1(2,n), C_1(3,n), C_1(5,n), \ldots, C_1(3203,n),5),
$$
is given by the Ramsey number, where 
\[C_1(2,n)=31+15.476\log |n|, \quad C_1(3,n)=2|n|^{17}+6, \quad C_1(k,n)=2|n|^5+3 \quad \text{for  } k \geq 5.\] Note $R(n)\gg (\sqrt{2})^{|n|^5}$ by a classical result of Erd\"os~\cite{E47} on lower bounds of Ramsey numbers, thus their upper bound on $M_{\leq \infty}(n)$ is at least exponential in $|n|$. Our next theorem provides a dramatic improvement on their bound.

\begin{thm}\label{thm:ABC}
Assume the ABC conjecture. If $n$ is a nonzero integer, then $M_{\leq \infty}(n)\ll \widetilde{f}(2|n|^{17})$. In particular, there is an absolute constant $L$, such that $M_{\leq \infty}(n)\ll \exp\big(L(\log \log |n|)^2\big)$.
\end{thm}

The following corollary follows immediately by combining Theorem~\ref{thm:ABC} with Theorem~\ref{thm:LPS} or Theorem~\ref{thm:absolutedf}.

\begin{cor}\label{cor:ABCLPS}
Assume the ABC conjecture and Conjecture~\ref{conj:LPS2}. Then $M_{\leq \infty}(n)\ll (\log |n|)^{4}$; and $M_{\leq \infty}(n)\ll (\log |n|/\log \log |n|)^{2}$ if we further assume the Uniformity Conjecture.
\end{cor}

\section{Preliminaries}\label{sec:prelim}
\subsection{Tools from extremal graph theory}\label{subsec:exgraph}
We first introduce a few basic terminologies from graph theory. A 
bipartite graph $G$ with bipartition $(A,B)$ (where $A \cap B=\emptyset$) is a graph with vertex set $A\cup B$ such that no two vertices in $A$ (resp. $B$) are adjacent. $K_r$ denotes the complete graph on $r$ vertices, namely, there is an edge between any pair of distinct vertices. $K_{s,t}$ denotes a complete bipartite graph with bipartition $(A,B)$, where $|A|=s$ and $|B|=t$, that is, there is an edge between $a$ and $b$ for all $a\in A$ and $b \in B$.

Below, we state two fundamental results in extremal graph theory. The first one is Tur\'an's theorem~\cite{T41} regarding forbidden complete subgraphs, and the second one is the K\"ov\'ari--S\'os--Tur\'an theorem \cite{KST54} concerning forbidden complete bipartite subgraphs.

\begin{lem}[Tur\'an's theorem]\label{lem:Turan}
Let $G$ be a graph with $n$ vertices such that it does not contain $K_r$ as a subgraph. Then the number of edges of $G$ is at most $\frac{1}{2} (1-\frac{1}{r-1}) n^2$.    
\end{lem}

\begin{lem}[K\"ov\'ari--S\'os--Tur\'an theorem]\label{lem:KST}
Let $G$ be a graph with $n$ vertices. If $G$ does not contain $K_{s,t}$ as a subgraph, where $s\leq t$, then the number of edges of $G$ is at most $(t-1)^{1/s} n^{2-1/s}+(s-1)n$.   
\end{lem}

In the next remark, we explain how a non-existence result of bipartite Diophantine tuples could be combined with the K\"ov\'ari--S\'os--Tur\'an theorem to address Question~\ref{q}. 

\begin{rem}\label{rem:q}
Let $k\geq 2$ and $n \neq 0$. Let $\delta \in (0,1]$ be a real number and let $X$ be a subset of positive integers, such that among all subsets $\{a,b\}$ of $X$ with size $2$, there are at least $\delta \binom{|X|}{2}$ pairs of them such that $ab+n$ is a perfect $k$-th power. Then we can build a graph $G$ with vertex set $X$, such that two distinct vertices $a,b\in X$ are adjacent if and only if $ab+n$ is a perfect $k$-th power. Then the given condition is equivalent to the condition that the number of edges of $G$ is at least $\delta \binom{|X|}{2}$.

Suppose that we can show that there is no bipartite Diophantine tuple $(A,B)$ with property $BD_k(n)$, where $2\leq s=|A|\leq |B|=t$ (in our actual applications, $s$ and $t$ might not be absolute constants, and could be slowly growing functions of $N$, where $A,B \subseteq[N]$). Then in the graph theory language, this means that $G$ does not contain $K_{s,t}$ as a subgraph, and thus \cref{lem:KST} implies that the number of edges of $G$ is at most $(t-1)^{1/s} |X|^{2-1/s}+(s-1)|X|$. By comparing the lower and upper bounds on the number of edges of $G$, we have
$$
\delta \binom{|X|}{2} \leq (t-1)^{1/s} |X|^{2-1/s}+(s-1)|X|.
$$
It follows that $|X|\leq (2/\delta)^s(t-1)+2s/\delta$. 
\end{rem}

The K\"ov\'ari--S\'os--Tur\'an theorem is essentially a consequence of the Cauchy–Schwarz inequality. For our purpose, we also need the following two helpful variants of Lemma~\ref{lem:KST}.

\begin{lem}[{\cite[Lemma 4]{BG04}}]\label{lem:KST2}
Let $G$ be a bipartite graph with bipartition $(A,B)$, where $|A|=n\leq |B|=m$ and the vertices of $G$ are labeled by positive integers. Suppose that for each $X \subseteq A$ and $Y \subseteq B$ (and for each $X \subseteq B$ and $Y \subseteq A$) with $|X|=r$ and $|Y|=t$ such that $\max_{x \in X} x<\min_{y \in Y} y$, the induced subgraph $G[X \cup Y]$ is not complete bipartite. Then the number of edges of $G$ is at most $2(t-1)^{1/r} mn^{1-1/r}+2(r-1)m$. 
\end{lem}

\begin{lem}[{\cite[Lemma 2.4]{GS07}}]\label{lem:KST3}
Let $G$ be a graph with $n$ vertices, with the edges colored by $k$ colors. Suppose that $G$ does not contain a cycle through vertices $v_1,v_2,v_3,v_4$ such that the edges $v_1v_2$ and $v_4v_1$ have the same color, and the edges $v_2v_3$ and $v_3v_4$ have the same color. Then the number of edges of $G$ is at most $k^{1/2}n^{3/2}+kn$.    
\end{lem}

\subsection{Larger sieve}
We will use sieve methods in our proofs. In particular, we will use a few variants of the larger sieve. We first recall Gallagher's larger sieve \cite{G71}.
\begin{lem}[Gallagher's larger sieve]\label{lem:GS}  
Let $N \in \N$ and $A\subseteq\{1,2,\ldots, N\}$. Let ${\mathcal P}$ be a set of primes. For each prime $p \in {\mathcal P}$, let $A_p=A \pmod{p}$. For any $1<Q\leq N$, we have
$$
 |A|\leq \frac{\underset{p\leq Q, p\in \mathcal{P}}\sum\log p - \log N}{\underset{p\leq Q, p \in \mathcal{P}}\sum\frac{\log p}{|A_p|}-\log N},
$$
provided that the denominator is positive.
\end{lem}

Croot and Elsholtz \cite{CE04} proved several variants of Gallagher's larger sieve. In particular, the following variant refines Lemma~\ref{lem:GS} in certain ranges. Indeed, in the proof of Theorem~\ref{thm:thm1}, this variant outweighs the original version.

\begin{lem}[{\cite[Theorem 3]{CE04}}]\label{lem:GS2}
Let $N \in \N$ and $A\subseteq\{1,2,\ldots, N\}$. Let ${\mathcal P} \subseteq [2,Q]$ be a set of primes. For each prime $p \in {\mathcal P}$, let $A_p=A \pmod{p}$. For any $1<Q\leq N$, we have
$$
 |A|\leq \max\biggl\{Q, \frac{23N \exp\big(\sum_{p\in \mathcal{P}}\frac{\log p}{p}\big)}{\exp\big(\sum_{p\in \mathcal{P}}\frac{\log p}{|A_p|}\big)}\biggr\}.
$$
\end{lem}

When one applies larger sieves, one usually needs to apply the prime number theorem for arithmetic progressions. However, in our applications, the modulus is not always fixed, so we need the following quantitative version of Linnik's theorem; see, for example \cite[Corollary 18.8]{IK04}. 

\begin{lem}[Quantitative Linnik's theorem]\label{lem:Linnik}
There exist positive constants $q_0$ and $L$ such that whenever $q\geq q_0$ is an integer, $x\geq q^L$, and $a$ is an integer with $\gcd(a,q)=1$, then 
$$
\sum_{\substack{p \leq x\\ p \equiv a \pmod q}} \log p\gg \frac{x}{\phi(q)\sqrt{q}},
$$
where the implied constant is absolute.
\end{lem}

Combining Lemma~\ref{lem:Linnik} with the prime number theorem for arithmetic progressions with modulus $q\leq q_0$ bounded, we obtain the following corollary.

\begin{cor}\label{cor:Linnik}
There exists positive constants $c$ and $L$ such that whenever $q$ is a positive integer, $x\geq q^L$, and $a$ is an integer with $\gcd(a,q)=1$, then 
$$
\sum_{\substack{p \leq x\\ p \equiv a \pmod q}} \log p \geq \frac{cx}{\phi(q)\sqrt{q}}.
$$
\end{cor}


The following auxiliary estimate will be useful.

\begin{lem}\label{lem:badprime}
Let $\mathcal{P}$ be a finite set of primes. Then we have
$\sum_{p \in \mathcal{P}} \frac{\log p}{p}\ll \log |\mathcal{P}|$.    
\end{lem}
\begin{proof}
Let $\mathcal{P}=\{q_1,q_2, \ldots, q_m\}$, and let $p_1<p_2<\ldots<p_m$ be the first $m$ odd primes. Since the function $\frac{\log x}{x}$ is decreasing when $x \geq 3$, it follows that 
\[
\sum_{p  \in \mathcal{P}} \frac{\log p}{p}=\sum_{j=1}^{m} \frac{\log q_j}{q_j} \leq \frac{\log 2}{2}+\sum_{p \leq p_m} \frac{\log p}{p}\ll \log p_m\ll \log (m\log m)\ll \log |\mathcal{P}|.\qedhere
\]
\end{proof}

\section{Contribution of large elements}\label{sec:large}
In this section, we deduce upper bounds on the number of ``large" elements in Diophantine tuples with property $D_{\leq d}(n)$. Our proofs are inspired by various ideas in \cite{BG04, GSS02, GS07, Y24+, Y26}.

We begin by collecting some known results regarding upper bounds on the number of ``large" elements in Diophantine tuples with property $D_k(n)$.

The following lemma is due to Dujella~\cite{D02}.
\begin{lem}[Dujella]\label{lem:D_2}
Let $n$ be a nonzero integer. If $A \subseteq [|n|^3,+\infty)$ is a Diophantine tuple with property $D_2(n)$, then $|A|\leq 21$.
\end{lem}

The following lemma is a special case of \cite[Proposition 4.1]{Y24} due to the second author.

\begin{lem}[\cite{Y24}]\label{lem:K23}
Let $k \geq 3$ and $n$ be a nonzero integer. Let $\ell=3$ if $k=3$, and $\ell=2$ if $k\geq 4$. There do not exist integers $a_1,a_2,\ldots, a_{\ell}$ and $b_1, b_2, \ldots, b_{s_k+1}$, such that $a_1<a_2<\ldots<a_{\ell}\leq b_1<b_2<\ldots<b_{s_k+1}$, $2|n|^{t_k} \leq b_1$,  and $a_ib_j+n$ is a perfect $k$-th power for all $1\leq i \leq \ell$ and $1\leq j \leq s_k+1$.    
\end{lem}

\begin{lem}[{\cite[Proposition 4.3]{Y24}}]\label{lem:large}
Let $k \geq 3$ and let $n$ be a nonzero integer. $(A,B)$ be a bipartite Diophantine tuple with property $BD_{k}(n)$. If $\min \{|A|, |B|\} \geq r_k$, then in both sets $A$ and $B$, at most $s_k$ elements are at least $2|n|^{t_k}$.
\end{lem}

Next, we use~\cref{lem:K23} and Lemma~\ref{lem:KST2} to deduce the following proposition. 

\begin{prop}\label{prop:twoset}
Let $k\geq 3$ and let $n$ be a nonzero integer. Let $A,B$ be two finite sets of positive integers in $[2|n|^{17}, \infty)$ with $|A|\leq |B|$. Then the number of pairs $(a,b)$ such that $a\in A$, $b \in B$ and $ab+n$ is is a perfect $k$-th power is at most $7|B||A|^{1/2}$ when $k \geq 4$, and at most $8|B||A|^{2/3}$ when $k=3$.
\end{prop}
\begin{proof}
Note that $t_k<17$ and $s_k\leq 6$ for all $k \geq 3$. We first consider the case $k \geq 4$. In this case, the corollary follows immediately from Lemma~\ref{lem:KST2} and~\cref{lem:K23} by building a bipartite graph with bipartition $(A,B)$ such that $a\in A$ and $b \in B$ are adjacent if and only if $ab+n$ is a perfect $k$-th power and the inequality $2\sqrt{6}|B||A|^{1/2}+2|B|\leq 7|B||A|^{1/2}$. The proof of the case $k=3$ is similar; Lemma~\ref{lem:KST2} implies that the number of edges is at most $2 \cdot 6^{1/3} |B||A|^{2/3}+4|B|<8|B||A|^{2/3}$.
\end{proof}

\begin{cor}\label{cor:contribution}
Let $k\geq 2$ and let $n$ be a nonzero integer. If $A$ is a finite set of positive integers in $[2|n|^{17}, \infty)$, then the number of pairs $\{a,b\}$ such that $a,b\in A$ with $a\neq b$ and $ab+n$ is is a perfect $k$-th power is at most $7|A|^{3/2}$ when $k \geq 4$, at most $8|A|^{5/3}$ when $k=3$, and at most $10|A|^2/21$ when $k=2$.   
\end{cor}
\begin{proof}
When $k\geq 3$, the corollary follows from Proposition~\ref{prop:twoset} by setting $B=A$.    

It remains to prove the case $k=2$. Let $G$ be the graph with vertex set $A$ such that two distinct vertices $a,b$ are adjacent if and only if $ab+n$ is a square. By Lemma~\ref{lem:D_2}, $G$ does not contain $K_{22}$ as a subgraph. It follows from Lemma~\ref{lem:Turan} that the number of edges of $G$ is at most $10|A|^2/21$.
\end{proof}

The next proposition bounds the number of large elements in a Diophantine tuple with property $D_{\leq d}(n)$, with $d$ finite. 

\begin{prop}\label{prop:large}
Let $n,d$ be integers with $n \neq 0$ and $d\geq 2$. If $A \subseteq [2|n|^{17}, \infty)$ is a Diophantine tuple with property $D_{\leq d}(n)$, then $|A|\ll (d/\log d)^2$, where the implied constant is absolute.    
\end{prop}
\begin{proof}
Since $A$ satisfies property $D_{\leq d}(n)$, for each $a,b\in A$ with $a\neq b$, there is a prime $p\leq d$ such that $ab+n$ is a perfect $p$-th power. It follows from Corollary~\ref{cor:contribution} that
$$
\binom{|A|}{2}\leq \frac{10|A|^2}{21}+8|A|^{5/3}+7|A|^{3/2} \cdot (\pi(d)-2).
$$
Thus $|A|\ll (\pi(d))^2$, as required.
\end{proof}

We also need to bound the number of large elements in a Diophantine tuple with property $D_{\leq \infty}(n)$. To achieve that, we apply tools from Diophantine approximation. We recall the following fundamental result in linear forms of logarithms of algebraic numbers; see, for example, \cite{BW93}. For a nonzero rational number $\alpha=a/b$, where $a,b$ are coprime integers, its height $H(\alpha)$ is defined as $\max \{|a|,|b|\}$.

\begin{lem}[\cite{BW93}]\label{lem:linearform}
Let $b_1$ and $b_2$ be non-zero integers and let $\alpha_1$ and $\alpha_2$ be positive rational numbers. Put $A_i=\max \{2, H(\alpha_i)\}$ for $i \in \{1,2\}$, $B=\max \{|b_1|, |b_2|,2\}$, and $\Lambda=b_1\log \alpha_1+b_2\log \alpha_2$. Then there exists an effectively computable positive constant $C$ such that if $\Lambda \neq 0$, then
$$
\log |\Lambda|>-C\log A_1 \cdot \log A_2 \cdot \log B.
$$ 
\end{lem}

\begin{prop}\label{prop:p1p2}
Let $n$ be a nonzero integer and let $M \geq (4|n|)^{17}$. There is an absolute constant $C'$ such that if $a_1, a_2, a_3,a_4$ are distinct integers in $[\sqrt{M}, M]$, then there do not exist two primes $p_1, p_2 \geq (C'\log M \log \log M)^{1/2}$ such that there exist positive integers $x_1,x_2,x_3,x_4$ with
\begin{equation}\label{eq:1234}
a_1a_2+n=x_1^{p_1}, \quad a_2a_3+n=x_2^{p_2},\quad  a_3a_4+n=x_3^{p_2}, \quad a_4a_1+n=x_4^{p_1}.
\end{equation}
\end{prop}
\begin{proof}
Let $C$ be the constant from Lemma~\ref{lem:linearform}. We show that the proposition is true with $C'=\max \{1,21C\}$.

For the sake of contradiction, suppose that there exist primes $p_1, p_2 \geq (C'\log M \log \log M)^{1/2}$ and positive integers $x_1,x_2,x_3,x_4$ satisfying equation~\eqref{eq:1234}. Let $t=\min \{p_1,p_2\}$. Without loss of generality, assume that $x_1^{p_1}=\max \{x_1^{p_1}, x_2^{p_2}, x_3^{p_2}, x_4^{p_1}\}$. Note that $x_3^{p_2}=a_3a_4+n\geq M-|n|>M/2$. Also, we have $$
\max \{x_1,x_2,x_3,x_4)\leq (M^2+|n|)^{1/t}\leq M^{3/t},$$
$$\max \{p_1,p_2\}\leq \log_2(M^2+|n|)\leq 3\log_2 M<4.5\log M.
$$

Observe that 
$$
(x_1^{p_1}-n)(x_3^{p_2}-n)=a_1a_2a_3a_4=(x_2^{p_2}-n)(x_4^{p_1}-n).
$$
It follows that
\begin{align*}
x_1^{p_1}x_3^{p_2}-x_2^{p_2}x_4^{p_1}
&=n(x_1^{p_1}+x_3^{p_2}-x_2^{p_2}-x_4^{p_1})\\
&=n(a_1a_2+a_3a_4-a_2a_3-a_4a_1)=n(a_1-a_3)(a_2-a_4)\neq 0.
\end{align*}
Thus,
$$
0<\bigg|\frac{x_2^{p_2}x_4^{p_1}}{x_1^{p_1}x_3^{p_2}}-1\bigg|\leq \frac{2|n|x_1^{p_1}}{x_1^{p_1}x_3^{p_2}}\leq \frac{2|n|}{x_3^{p_2}}\leq \frac{4|n|}{M}\leq \frac{M^{1/17}}{M}=M^{-\frac{16}{17}},
$$
where we used the assumption that $M\geq (4|n|)^{17}$. 

Set
$$
\Lambda=\log \frac{x_2^{p_2}x_4^{p_1}}{x_1^{p_1}x_3^{p_2}}=p_1 \log \bigg(\frac{x_4}{x_1}\bigg)+p_2\log \bigg(\frac{x_2}{x_3}\bigg);
$$
then $\Lambda \neq 0$ and $|e^{\Lambda}-1|\leq M^{-16/17}\leq 4^{-16}$. It is easy to verify that if $z$ is a real number such that $|e^z-1|<1/8$, then $|z|<1/2$; and if $|z|<1/2$, then $|e^z-1|\geq |z|/2$. It follows that $|\Lambda|\leq 2M^{-16/17}$ and thus
$$\log |\Lambda| \leq \log 2-\frac{16}{17}\log M<-\frac{15}{17}\log M.$$ On the other hand, we can apply Lemma~\ref{lem:linearform} with $$\alpha_1=x_4/x_1, \quad \alpha_2=x_2/x_3, \quad b_1=p_1, \quad b_2=p_2$$ to get a lower bound on $\log |\Lambda|$. Note that we have $$\log A=\max \{\log H(\alpha_1), \log H(\alpha_2)\}\leq (3\log M)/t,$$  $$\log B=\log \max\{p_1,p_2\}<\log(4.5\log M)<2\log \log M.$$ Then Lemma~\ref{lem:linearform} implies that
$$
-\frac{15}{17}\log M>\log |\Lambda|>-C \frac{9\log^2 M}{t^2} (2\log \log M).
$$
It follows that
$$
t<\sqrt{21C \log M \log \log M},
$$
which contradicts the assumption that $t\geq \sqrt{C'\log M\log \log M}$ since $C'\geq 21C$.
\end{proof}

Now we are ready to prove a key result of the section.

\begin{prop}\label{prop:largeoo}
Let $n$ be a nonzero integer and let $N\geq 4|n|^{17}$. If $A \subseteq [(4|n|)^{17},N]$ is a Diophantine tuple with property $D_{\leq \infty}(n)$, then $|A|\ll \log N$, where the implied constant is absolute. 
\end{prop}
\begin{proof}
Let $M_1=4|n|^{17}$, and inductively we define $M_{i}=M_{i-1}^2$ for each integer $i\geq 2$. Let $A_i=A \cap [M_i, M_{i+1}]$ for each positive integer $i$; then we have $A=\bigcup_{i=1}^{\ell} A_i$, where $\ell \ll \log \log N$. Thus, it suffices to show that $|A_i|\ll \frac{\log N}{\log \log N}$ for each $i$.

Fix an integer $i$ with $1\leq i \leq \ell$. Since $A_i$ is a Diophantine tuple with property $D_{\leq \infty}(n)$, for each pair $\{a,b\} \subseteq A_i$ with $a\neq b$, $ab+n$ is a perfect $p$-th power for some prime $p$; however, since $ab+n\leq N^2+n<4N^2$, it is necessary that $p\leq \log_2(4N^2)\leq 4\log N$. Let $C'$ be the constant from~\cref{prop:p1p2}. By Corollary~\ref{cor:contribution}, the number of such pairs corresponding to primes $p\leq t:=(C'\log N \log \log N)^{1/2}$ is at most
$$
\frac{10|A_i|^2}{21}+8|A_i|^{5/3}+7|A|^{3/2} \pi(t).
$$
Next, we give an upper bound on the number of pairs corresponding to primes $t<p\leq 4\log N$ using Lemma~\ref{lem:KST3}. Build a graph $G_i$ with vertex set $A_i$. If $a,b\in A_i$ with $a\neq b$ such that $ab+n$ is a $p$-th power for some prime $p \in (t,4\log N]$; then there is an edge between $a$ and $b$ and we color it using the smallest such $p$. Then the edges of $G_i$ are colored by at most $\pi(4\log N)$ colors. By Proposition~\ref{prop:p1p2}, $G_i$ does not contain a cycle through vertices $v_1,v_2,v_3,v_4$ such that the edges $v_1v_2$ and $v_4v_1$ have the same color, and the edges $v_2v_3$ and $v_3v_4$ have the same color. Thus, Lemma~\ref{lem:KST3} implies that the number of edges in $G_i$ is at most
$$
\sqrt{\pi(4\log N)} |A_i|^{3/2}+\pi(4\log N) |A_i|.
$$
It follows that
$$
\binom{|A_i|}{2}\leq \frac{10|A_i|^2}{21}+8|A_i|^{5/3}+7|A|^{3/2} \pi(t)+ \sqrt{\pi(4\log N)} |A_i|^{3/2}+\pi(4\log N) |A_i|.
$$
Hence $|A_i|\ll \pi(t)^2+\pi(4\log N)\ll \frac{\log N}{\log \log N}$, as required.
\end{proof}

\section{Finite field models}\label{sec:finitefield}
In this section, we use character sum estimates to study bounds on various finite field analogues of Diophantine tuples as a preparation to apply sieve methods in Section~\ref{sec:bipartite}. Throughout, let $\F_p$ be the finite field with $p$ elements, and $\F_p^*=\F_p \setminus \{0\}$. We remark that some finite field analogues of Diophantine tuples have been studied in, for example, \cite{G01, KYY, S14}.

\subsection{Diophantine tuples with property $D_{\leq d}(n)$}
The following Vinogradov-type double character sum estimate is well-known; see, for example \cite[Theorem 7]{G01} and \cite[Proposition 3.1]{KYY}.

\begin{lem}[Vinogradov]\label{lem:Vinogradov}
Let $\chi$ be a non-trivial multiplicative character of $\F_p$ and $\lambda \in \F_p^*$. For any $A,B \subseteq \F_p^*$,  we have
 $$
 \bigg|\sum_{a\in A,\, b\in B}\chi(ab+\lambda)\bigg|  \leq \sqrt{p|A||B|}.
 $$    
\end{lem}

Next, we use Lemma~\ref{lem:Vinogradov} to deduce an upper bound on the finite field analogue of a Diophantine tuple with property $D_{\leq d}(n)$.

\begin{prop}\label{prop:ap}
Let $p_1,p_2,\ldots, p_m$ be distinct primes and let $p$ be a prime such that $p \equiv 1 \pmod {p_i}$ for each $1\leq i \leq m$. Let $\lambda \in \F_p^*$. For each $1\leq i \leq m$, let $H_i=\{x^{p_i}: x \in \F_p\}$. If $A$ a subset of $\F_p^*$ such that $ab+\lambda \in \bigcup_{i=1}^{m} H_i$ for all $a,b \in A$ with $a\neq b$, then
$$|A|\leq (2^m\sqrt{p}+2) \cdot \prod_{i=1}^m \bigg(1-\frac{1}{p_i}\bigg)^{-1}. $$
\end{prop}
\begin{proof}
For a set $S\subseteq \F_p$, we use $S^*$ to denote $S\setminus \{0\}$, and $\mathbf{1}_S$ to denote the indicator function of the set $S$. For each $1\leq i \leq m$, let $\chi_{p_i}$ be a multiplicative character of $\F_p$ of order $(p-1)/p_i$; then by the orthogonality relation, we have 
\begin{equation}\label{eq:orth}
\mathbf{1}_{H_i^*}=\frac{1}{p_i} \sum_{j_i=0}^{p_i-1} \chi_{p_i}^{j_i}.
\end{equation}

Let $H=\bigcup_{i=1}^m H_i$. Since $ab+\lambda \in H$ for all $a,b \in A$ with $a\neq b$, and for each $a\in A$, there is at most one $b \in A$ such that $ab+\lambda=0$, it follows that
\begin{equation}\label{eq:lb}
\sum_{a,b\in A} \mathbf{1}_{H^*}(ab+\lambda) \geq \sum_{a,b\in A} \mathbf{1}_{H}(ab+\lambda)-|A|\geq |A|^2-2|A|.    
\end{equation}

On the other hand, by equation~\eqref{eq:orth}, we have 
\begin{equation}\label{eq:sum1}
\mathbf{1}_{H^*}=1-\prod_{i=1}^m (1-\mathbf{1}_{H_i^*})=1-\prod_{i=1}^m \bigg(1-\frac{1}{p_i} \sum_{j=0}^{p_i-1} \chi_{p_i}^j\bigg)=-\sum_{\substack{I \subseteq [m]\\|I|\neq \emptyset}} \prod_{i \in I} \bigg(-\frac{1}{p_i} \sum_{j_i=0}^{p_i-1} \chi_{p_i}^{j_i}\bigg).
\end{equation}
Fix a nonempty subset $I$ of $[m]$. Let $|I|=k$ and write $\{p_i:i \in I\}=\{q_1,q_2,\ldots, q_{k}\}$. Then we have
\begin{equation}\label{eq:sum2}
\prod_{i \in I} \bigg(\sum_{j_i=0}^{p_i-1} \chi_{p_i}^{j_i}\bigg)=\sum_{j_1=0}^{q_1-1} \cdots \sum_{j_k=0}^{q_k-1} \prod_{i=1}^{k} \chi_{q_i}^{j_i};    
\end{equation}
note that the character $\prod_{i=1}^{k} \chi_{q_i}^{j_i}$ is trivial if and only if $j_1=j_2=\cdots=0$. It then follows from Lemma~\ref{lem:Vinogradov} that
\begin{equation}\label{eq:sum3}
\sum_{a,b\in A} \prod_{i \in I} \bigg(\sum_{j_i=0}^{p_i-1} \chi_{p_i}^{j_i}(ab+\lambda)\bigg)=\sum_{j_1=0}^{q_1-1} \cdots \sum_{j_k=0}^{q_k-1} \sum_{a,b \in A} \bigg(\prod_{i=1}^{k} \chi_{q_i}^{j_i}\bigg) (ab+\lambda)=|A|^2+E(I),
\end{equation}
where $|E(I)|\leq \prod_{i\in I}p_i \cdot \sqrt{p}|A|.$

Combining equations~\eqref{eq:sum1},~\eqref{eq:sum2}, and~\eqref{eq:sum3}, 
\begin{align}
\sum_{a,b\in A} \mathbf{1}_{H^*}(ab+\lambda)
&=- \sum_{\substack{I \subseteq [m]\\I\neq \emptyset}} \sum_{a,b \in A}\prod_{i \in I} \bigg(-\frac{1}{p_i} \sum_{j_i=0}^{p_i-1} \chi_{p_i}^{j_i}(ab+\lambda)\bigg)\notag \\
&=- \sum_{\substack{I \subseteq [m]\\I\neq \emptyset}} \bigg(\frac{|A|^2}{\prod_{i\in I}(-p_i)}+\frac{E(I)}{\prod_{i\in I}(-p_i)}\bigg)\notag \\
&\leq\bigg(1-\prod_{i=1}^m \bigg(1-\frac{1}{p_i}\bigg)\bigg)|A|^2+2^m \sqrt{p}|A|. \label{eq:ub}
\end{align}
Comparing inequalities~\eqref{eq:lb} and~\eqref{eq:ub}, we get
$$
(2^m\sqrt{p}+2)|A|\geq \prod_{i=1}^m \bigg(1-\frac{1}{p_i}\bigg) \cdot |A|^2,
$$
as required.
\end{proof}

\subsection{Bipartite Diophantine tuples}
Next, we consider finite field analogues of bipartite Diophantine tuples. We begin by applying Weil's bound on complete character sums. 

\begin{lem}\label{lem:Weil}
Let $d \geq 2$ and $p$ be a prime such that $p \equiv 1 \pmod d$. Let $A \subseteq \F_p^*$, $B \subseteq \F_p$, and $\lambda \in \F_p^*$, such that $ab+\lambda \in \{x^d: x \in \F_p\}$ for all $a\in A$ and $b\in B$. If $|A|=m$, then $|B|\leq \frac{p}{d^m}+m\sqrt{p}$.
\end{lem}
\begin{proof}
Let $\chi$ be a multiplicative character of $\F_p$ with order $d$. Let $A=\{a_1, a_2, \ldots, a_m\}$. Note that for each $b \in B$ and each $1\leq i \leq m$, we have $\chi(a_ib+\lambda)=1$ and thus $\chi(b+a_i^{-1}\lambda)=\overline{\chi(a_i)}$. In particular, each $b \in B$ is a solution to the system of equations $$\chi(x+a_i^{-1}\lambda)=\overline{\chi(a_i)}, \quad \forall 1\leq i \leq m.$$ 
Now, a classical application of Weil's bound (see, for example, \cite[Exercise 5.66]{LN97}) implies that 
$$
|B|\leq \frac{p}{d^m}+\bigg(m-1-\frac{m}{d}+\frac{1}{d^m}\bigg)\sqrt{p}+\frac{m}{d}<\frac{p}{d^m}+m\sqrt{p},
$$
as required.
\end{proof}


Next, we consider an explicit version of a variant of a double character sum estimate due to Karatsuba \cite{K91} (see also \cite[Lemma 2.2]{S13}). For our applications, we establish an explicit dependence on the parameter $\nu$ for the upper bound below. Karatsuba's estimate is better than Vinogradov's estimate when $A$ and $B$ are asymmetric in the sense that the sizes of $A$ and $B$ are not comparable.

\begin{lem}\label{lem:charsum}
 Let $A \subseteq \F_p, B \subseteq \F_p^*$, and $\lambda \in \F_p^*$. Then for any non-trivial multiplicative character $\chi$ of $\F_p$ and any positive integer $\nu$, we have
$$
\bigg|\sum_{\substack{a\in A\\ b\in B}} \chi(ab+\lambda)\bigg|\leq |B|^{(2\nu-1)/2\nu} \big(2\nu|A|^{2\nu} \sqrt{p} + (2\nu)^\nu |A|^\nu p\big)^{1/2\nu}.
$$   
\end{lem}

\begin{proof}
By H\"older's inequality, we have
\begin{align*}
\bigg|\sum_{\substack{a\in A\\ b\in B}} \chi(ab+\lambda) \bigg|
&=\bigg|\sum_{b \in B} \chi(b) \sum_{a \in A} \chi(a+b^{-1}\lambda)\bigg|\\  
&\leq \bigg(\sum_{b \in B} |\chi(b)|^{2\nu/(2\nu-1)}\bigg)^{(2\nu-1)/2\nu} \bigg(\sum_{b \in B} \bigg|\sum_{a \in A} \chi(a+b^{-1}\lambda)\bigg|^{2\nu}\bigg)^{1/2\nu}\\
&\leq |B|^{(2\nu-1)/2\nu} \bigg(\sum_{x \in \F_p} \bigg|\sum_{a \in A} \chi(a+x)\bigg|^{2\nu}\bigg)^{1/2\nu}.
\end{align*}

Note that
\begin{align}\label{eq: summand}
\sum_{x \in \F_p} \bigg|\sum_{a \in A} \chi(a+x)\bigg|^{2\nu}
=\sum_{a_1, a_2, \ldots, a_{2\nu} \in A} \bigg(\sum_{x \in \F_p} \chi\bigg(\prod_{j=1}^{\nu} (x+a_j)\bigg) \cdot 
\overline{\chi}\bigg(\prod_{k=\nu+1}^{2\nu} (x+a_k)\bigg)\bigg),
\end{align}
where the sum on the right-hand side runs over all $2\nu$-triples $(a_1, a_2, \ldots, a_{2\nu}) \in A^{2\nu}$. If there exists $i$ with $a_i \neq a_j$ for each $j \neq i$, then a standard application of Weil's bound implies that the second summand on the right-hand side of equation~\eqref{eq: summand} is $2\nu \sqrt{p}$ (see for example \cite[Corollary 11.24]{IK04}); otherwise the summand is trivially bounded by $p$. Note that in the latter case, each $a_i$ appears at least twice, so the number of such tuples $(a_1, a_2, \ldots, a_{2\nu}) \in A^{2\nu}$ is at most $\binom{2\nu}{\nu} \nu! |A|^\nu=\frac{(2\nu)!}{\nu!} |A|^\nu$. Therefore, 
$$
\sum_{x \in \F_p} \bigg|\sum_{a \in A} \chi(a+x)\bigg|^{2\nu} \leq 2\nu|A|^{2\nu} \sqrt{p} + \frac{(2\nu)!}{\nu!} |A|^\nu p<2\nu|A|^{2\nu} \sqrt{p} + (2\nu)^\nu |A|^\nu p,
$$
and the required estimate follows.
\end{proof}

\begin{cor}\label{cor:Kb}
Let $k\geq 2$ be an integer and $p$ be a prime such that $\gcd(k,p-1)>1$. Let $A,B \subseteq \F_p$ and $\lambda \in \F_p^*$, such that $ab+\lambda \in  \{x^k: x \in \F_p\}$ for all $a\in A$ and $b\in B$. If $\nu$ is a positive integer such that $|A|\geq 2\nu p^{1/2\nu}$, then $|B|\leq 12\nu\sqrt{p}$.
\end{cor}
\begin{proof}
Note that $\{x^k: x \in \F_p\}=\{x^{\gcd(k,p-1)}: x \in \F_p\}$. Let $\chi$ be a multiplicative character of $\F_p$ with order $\gcd(k,p-1)$ and let $B'=B \setminus \{0\} \subseteq \F_p^*$. Note that for each $b \in B'$, there is at most one $a \in A$ such that $ab+\lambda=0$. Thus, Lemma~\ref{lem:charsum} implies that 
\begin{equation}\label{eq:AB}    
|B'|(|A|-1)\leq \sum_{\substack{a\in A\\ b\in B'}} \chi(ab+\lambda)\leq |B'|^{(2\nu-1)/2\nu} (2\nu|A|^{2\nu} \sqrt{p} + (2\nu)^\nu |A|^\nu p)^{1/2\nu}.
\end{equation}
Since $|A|\geq 2\nu p^{1/2\nu}$, it follows that $2\nu|A|^{2\nu} \sqrt{p} > (2\nu)^\nu |A|^\nu p$. Therefore, inequality~\eqref{eq:AB} implies 
$$
|B'|(|A|-1)^{2\nu} \leq 4\nu |A|^{2\nu}\sqrt{p}.
$$
Since $|A|\geq 2\nu+1$, it follows that
$$
|B'|\leq 4\nu \sqrt{p} \cdot \bigg(\frac{|A|}{|A|-1}\bigg)^{2\nu}\leq 4\nu \sqrt{p} \cdot \bigg(1+\frac{1}{2\nu}\bigg)^{2\nu}<4e\nu \sqrt{p}<
12\nu \sqrt{p}-1.
$$
It follows that $|B|\leq |B'|+1\leq 12\nu\sqrt{p}$, as required.
\end{proof}

\section{Bounds on Bipartite Diophantine tuples}\label{sec:bipartite}

In this section, we use sieve methods to prove the results stated in Section~\ref{subsec:bipartite}. 

\subsection{Proof of Theorem~\ref{thm:thm1}}\ 

Let $N=2|n|^{t_k}$. Let $A,B \subseteq \N$ such that $AB+n \subseteq \{x^k: x \in \N\}$ with $|B|\geq |A|\geq m+s_k$. Then $|B|\geq |A|\geq m+s_k\geq (r_k-s_k)+s_k=r_k$. Let $A'=A \cap [1, N]$ and $B'=B \cap [1, N]$. It follows from Lemma~\ref{lem:large} that $|A'|\geq |A|-s_k\geq m$ and $|B'|\geq |B|-s_k\geq |B|-3$. We replace $A'$ with a subset of $A'$ of size $m$ if $|A'|>m$.

Consider the set of primes $\mathcal{P}=\{p: p\leq Q\}$, where $Q=N^{\phi(k)/\theta_{k,m}},$ and
\begin{equation}
\theta_{k,m}=\sum_{\substack{1\leq i \leq k\\\gcd(i,k)=1}}\gcd(i-1,k)^m    
\end{equation}
has been defined in equation~\eqref{eq:theta}. Note that $\phi(k)<k^m\leq \theta_{k,m}\leq k^{m+1}$. Since $m=o(\log \log |n|)$, as $|n|\to \infty$, we have $m^2k^{3m+1}=o(\log |n|)$, and thus 
\begin{equation}\label{eq:mk^{2m}}
mk^{2m}=o\bigg(\frac{\log N}{k^{m+1}}\bigg)=o(\log Q).    
\end{equation}

For each $p \in \mathcal{P}$, let $A_p$ be the image of $A'$ modulo $p$ and view $A_p$ as a subset of the finite field $\F_p$; similarly, define $B_p$. Let 
$$\mathcal{P}_1=\{p \in \mathcal{P}: |A_p|<m\}, \quad \mathcal{P}_2=\{p \in \mathcal{P}: 0 \in A_p\}, \quad
\mathcal{P}_3=\{p \in \mathcal{P}: p \mid n\}, \quad
$$
We claim that $|\mathcal{P}_1|\leq m^2\log_2N$. Indeed, if $p \in \mathcal{P}$ and $|A_p|<|A'|=m$, then there are two distinct elements $a_1,a_2 \in A'$ such that $p \mid (a_1-a_2)$. However, there are less than $m^2$ pairs of distinct elements $a_1, a_2 \in A'$, and for each such pair $a_1, a_2$, the number of distinct prime factors of $|a_1-a_2|$ is at most $\log_2 N$ since $0<|a_1-a_2|\leq N$. A similar argument shows that $|\mathcal{P}_2|\leq m \log_2 N$ and $|\mathcal{P}_3|\leq \log_2 N$.  

Let $\widetilde{\mathcal{P}}=\mathcal{P} \setminus \mathcal{P}_1\setminus \mathcal{P}_2 \setminus \mathcal{P}_3$. Next we apply Lemma~\ref{lem:GS2} to bound $|B|$ using the set of primes $\widetilde{\mathcal{P}}$. By \cite[Theorem 2.7]{MV07}, we have
$$
\sum_{p \in \widetilde{\mathcal{P}}} \frac{\log p}{p} \leq \sum_{p\leq Q} \frac{\log p}{p}=\log Q+O(1)=\frac{\phi(k)}{\theta_{k,m}} \log N+O(1),
$$
and thus
\begin{equation}\label{eq:eq0}
\exp\bigg(\sum_{p\in \widetilde{\mathcal{P}}}\frac{\log p}{p}\bigg) \ll Q.    
\end{equation}

Let $p \in \widetilde{\mathcal{P}}$; we claim that
\begin{equation}\label{eq:Bp}
|B_p|\leq \frac{p}{\gcd(k,p-1)^m}+m\sqrt{p}.
\end{equation}
Note that we have $p \nmid n$, $|A_p|=m$ and $0 \notin A_p$. If $\gcd(k,p-1)=1$, then we simply apply the trivial bound $|B_p|\leq p$; if $\gcd(k,p-1)>1$, then we have $$A_pB_p+n \subseteq \{x^k: x \in \F_p\}=\{x^{\gcd(k,p-1)}: x \in \F_p\}$$ with $p \nmid n$ and $A_p \subseteq \F_p^*$, and thus Lemma~\ref{lem:Weil} implies that inequality~\eqref{eq:Bp}.

It follows from inequality~\eqref{eq:Bp} that
\begin{equation}\label{eq:eq1}
\sum_{p \in \widetilde{\mathcal{P}}} \frac{\log p}{|B_p|}\geq \sum_{p \in \widetilde{\mathcal{P}}} \frac{\log p}{ \frac{p}{\gcd(k,p-1)^m}+m\sqrt{p}}= \sum_{p \in \widetilde{\mathcal{P}}} \frac{\log p}{ \frac{p}{\gcd(k,p-1)^m}}-\sum_{p \in \widetilde{\mathcal{P}}} \frac{\log p \cdot m\sqrt{p}}{ \frac{p}{\gcd(k,p-1)^m} (\frac{p}{\gcd(k,p-1)^m}+m\sqrt{p})}.    
\end{equation}
Note that from equation~\eqref{eq:mk^{2m}}, we have
\begin{align}
\sum_{p \in \widetilde{\mathcal{P}}} \frac{\log p \cdot m\sqrt{p}}{ \frac{p}{\gcd(k,p-1)^m} (\frac{p}{\gcd(k,p-1)^m}+m\sqrt{p})}
&\ll \sum_{p \in \widetilde{\mathcal{P}}} \frac{k^{2m}\log p \cdot m\sqrt{p}}{p^2} \notag\\
&\ll mk^{2m}\sum_{p} \frac{\log p}{p^{3/2}}\ll mk^{2m}=o(\log Q)  \label{eq:eq2}
\end{align}
as $|n|\to \infty$. On the other hand, by the prime number theorem for arithmetic progressions, we have
\begin{align}
\sum_{p\leq Q} \frac{\log p}{ \frac{p}{\gcd(k,p-1)^m}}
&=\sum_{\substack{1\leq i \leq k\\\gcd(i,k)=1}} \sum_{\substack{p\leq Q \\ p \equiv i \pmod k}}\frac{\log p}{ \frac{p}{\gcd(k,p-1)^m}} \notag\\
&=\sum_{\substack{1\leq i \leq k\\\gcd(i,k)=1}} \frac{(1+o(1))\gcd(k,p-1)^m}{\phi(k)} \log Q=\frac{(1+o(1))\theta_{k,m}}{\phi(k)} \log Q.\label{eq:eq3}
\end{align}
By Lemma~\ref{lem:badprime}, for each $j \in \{1,2,3\}$, we have
\begin{equation}\label{eq:eq4}
\sum_{p \in \mathcal{P}_j} \frac{\log p}{ \frac{p}{\gcd(k,p-1)^m}}\leq k^m \sum_{p \in \mathcal{P}_j} \frac{\log p}{p}\ll k^m \log |\mathcal{P}_j|\ll k^m \log \log N=o(\log Q).
\end{equation}
Thus, combining equation~\eqref{eq:eq3} and inequality~\eqref{eq:eq4}, we have 
\begin{equation}\label{eq:eq5}
\sum_{p \in \widetilde{\mathcal{P}}} \frac{\log p}{ \frac{p}{\gcd(k,p-1)^m}}\geq 
\sum_{p\leq Q} \frac{\log p}{ \frac{p}{\gcd(k,p-1)^m}}-
\sum_{j=1}^{3}\sum_{p \in \mathcal{P}_j} \frac{\log p}{ \frac{p}{\gcd(k,p-1)^m}}=\frac{(1+o(1))\theta_{k,m}}{\phi(k)} \log Q.
\end{equation}
Combining inequalities~\eqref{eq:eq1},~\eqref{eq:eq2}, and~\eqref{eq:eq5}, 
\begin{equation}\label{eq:eq6}
\sum_{p \in \widetilde{\mathcal{P}}} \frac{\log p}{|B_p|}\geq\frac{(1+o(1))\theta_{k,m}}{\phi(k)} \log Q. 
\end{equation}
Combining inequalities~\eqref{eq:eq0} and~\eqref{eq:eq6} with Lemma~\ref{lem:GS2}, 
$$
|B|\ll Q N^{o(1)}=N^{\frac{\phi(k)}{\theta_{k,m}}+o(1)},
$$
as required.

\subsection{Proof of Theorems~\ref{thm:loglogN2} and~\ref{thm:loglogN}}
The proofs of Theorems~\ref{thm:loglogN2} and~\ref{thm:loglogN} are inspired by several arguments used in \cite{G01, Y24} for bipartite Diophantine tuples, as well as the ``inverse sieve argument" developed by Elsholtz \cite{E01} in the inverse Goldbach problem.

\begin{proof}[Proof of Theorem~\ref{thm:loglogN2}]
Let $N=2|n|^{17}$. In view of \cref{lem:large}, we may assume that $A, B \subseteq [1,N]$. Let $\nu=\lceil \log \log N \rceil$. Set $$Q=(48\nu \phi(k) \log N)^2, \quad \mathcal{P}=\{p\leq Q: p \equiv 1 \pmod k, \, p \nmid n\}.$$ 
For each prime $p$, let $A_p$ be the image of $A$ modulo $p$, and define $B_p$ similarly. Let 
\[\mathcal{P}_1=\{p \in \mathcal{P}: |A_p|\geq 2\nu p^{1/2\nu}\}, \qquad \mathcal{P}_2=\mathcal{P}\setminus \mathcal{P}_1.\]

For each prime $p \in \mathcal{P}$, we can view $A_p$ and $B_p$ as subsets of $\F_p$. For each $a\in A$ and $b\in B$, $ab+n$ is a perfect $k$-th power. It follows that $A_pB_p+n \subseteq \{x^k: x \in \F_p\}$. Thus, by Corollary~\ref{cor:Kb}, if $p \in \mathcal{P}_1$, that is, $|A_p|\geq 2\nu p^{1/2\nu}$, then $|B_p|\leq 12\nu\sqrt{p}$.
Let
$$
S:= \sum_{\substack{p \in \mathcal{P}}} \log p, \quad S_1:= \sum_{\substack{p \in \mathcal{P}_1}} \log p, \quad S_2:= \sum_{\substack{p \in \mathcal{P}_2}} \log p.
$$
By the prime number for arithmetic progressions, we have 
\begin{equation}\label{eq:Sv2}
S\geq 
\sum_{\substack{p \equiv 1 \pmod k\\p \le Q}} \log p -\sum_{p \mid n} \log p \geq (1-o(1))\frac{Q}{\phi(k)}-\log N=(1-o(1))\frac{Q}{\phi(k)}.
\end{equation}
We consider the following two cases. 

Case 1: $S_1\geq S_2$. Then we have $S_1\geq S/2$ and thus
\begin{equation*}
\sum_{\substack{p \in \mathcal{P}_1}} \frac{\log p}{|B_p|} \geq \sum_{\substack{p \in \mathcal{P}_1}} \frac{\log p}{12\nu \sqrt{p}}\geq \frac{S_1}{12\nu \sqrt{Q}}\geq \frac{S}{24\nu\sqrt{Q}}.
\end{equation*}
It follows from inequality~\eqref{eq:Sv2} that
\begin{equation}\label{eq:case1v2}
\sum_{\substack{p \in \mathcal{P}_1}} \frac{\log p}{|B_p|}-\log N \geq \frac{S}{24\nu\sqrt{Q}}-\log N\geq (1-o(1))\frac{\sqrt{Q}}{24\nu \phi(k)}-\log N =(1-o(1))\log N.
\end{equation}
Applying Gallagher's sieve (Lemma~\ref{lem:GS}), and inequalities~\eqref{eq:Sv2} and~\eqref{eq:case1v2}, we conclude that
\[
|B|\leq \frac{\underset{p\in \mathcal{P}_1}\sum\log p - \log N}{\underset{p \in \mathcal{P}_1}\sum\frac{\log p}{|B_p|}-\log N}\leq (1+o(1))\frac{Q}{\phi(k)\log N}\ll_k \log N (\log \log N)^2 \ll \log |n| (\log \log |n|)^2.
\]

Case 2: $S_1< S_2$. Then we have $S_2\geq S/2$ and thus
\begin{equation}\label{eq:case2}
\sum_{\substack{p \in \mathcal{P}_2}} \frac{\log p}{|A_p|} \geq \sum_{\substack{p \in \mathcal{P}_2}} \frac{\log p}{2\nu p^{1/2\nu}}\geq \frac{S_2}{2\nu Q^{1/2\nu}}\geq \frac{S}{4\nu Q^{1/2\nu}}.
\end{equation}
Applying Gallagher's sieve (Lemma~\ref{lem:GS}), and inequalities~\eqref{eq:Sv2} and~\eqref{eq:case2}, we conclude that
\[
|A|\leq \frac{\underset{p\in \mathcal{P}_2}\sum\log p - \log N}{\underset{p \in \mathcal{P}_2}\sum\frac{\log p}{|A_p|}-\log N}\leq \frac{S}{\frac{S}{4\nu Q^{1/2\nu}}-\log N} \ll \nu Q^{1/2\nu} \ll_{k} \log \log N \ll \log \log |n|.
\]
Thus, we conclude that either $|A|\ll_{k} \log \log |n|$ or $|B|\ll_{k} \log |n| (\log \log |n|)^2$, as desired.
\end{proof}

Next, we use a similar idea to prove Theorem~\ref{thm:loglogN}.

\begin{proof}[Proof of Theorem~\ref{thm:loglogN}]
Let $\nu=\lceil \log \log N \rceil$. Let $c$ and $L$ be the positive constants from Corollary~\ref{cor:Linnik}. Without loss of generality, we may assume that $L \geq 10$ and $c\leq 1$. Let $$Q=k^L(24c^{-1}\nu\phi(k)\sqrt{k}\log N)^2, \quad \mathcal{P}=\{p\leq Q: p \equiv 1 \pmod k, \, p \nmid n\}.$$ 
We define $\mathcal{P}_1, \mathcal{P}_2, S, S_1, S_2$ in the same way as in the proof of Theorem~\ref{thm:loglogN2}.

Since $Q\geq k^L$, Corollary~\ref{cor:Linnik} implies that
\begin{equation}\label{eq:S}
S\geq 
\sum_{\substack{p \equiv 1 \pmod k\\p \le Q}} \log p -\sum_{p \mid n} \log p \geq \frac{cQ}{\phi(k)\sqrt{k}}-\log N.
\end{equation}

We consider the following two cases. 

Case 1: $S_1\geq S_2$. As in the proof of Theorem~\ref{thm:loglogN2}, we have $\sum_{\substack{p \in \mathcal{P}_1}} \frac{\log p}{|B_p|}\geq \frac{S}{24\nu\sqrt{Q}}$. It follows from inequality~\eqref{eq:S} that
\begin{equation}\label{eq:case1}
\sum_{\substack{p \in \mathcal{P}_1}} \frac{\log p}{|B_p|}-\log N \geq \frac{c\sqrt{Q}}{24\nu \phi(k)\sqrt{k}}-2\log N \geq (k^{L/2}-2)\log N \gg \log N.
\end{equation}
Applying Gallagher's sieve (Lemma~\ref{lem:GS}), inequalities~\eqref{eq:S} and~\eqref{eq:case1}, and the prime number theorem, we conclude that
\[
|B|\leq \frac{\underset{p\in \mathcal{P}_1}\sum\log p - \log N}{\underset{p \in \mathcal{P}_1}\sum\frac{\log p}{|B_p|}-\log N}\ll \frac{Q}{\log N}\ll k^{L+3} \nu^2 \log N \ll (\log N)^{L+5},
\]
where we used the assumption that $k\leq 5\log N$ in the last step, and the implied constants are all absolute. 

Case 2: $S_1< S_2$. As in the proof of Theorem~\ref{thm:loglogN2}, we have $\sum_{\substack{p \in \mathcal{P}_2}} \frac{\log p}{|A_p|}\geq \frac{S}{4\nu Q^{1/2\nu}}.$ Now Gallagher's sieve implies that
\[
|A|\leq \frac{\underset{p\in \mathcal{P}_2}\sum\log p - \log N}{\underset{p \in \mathcal{P}_2}\sum\frac{\log p}{|A_p|}-\log N}\leq \frac{S}{\frac{S}{4\nu Q^{1/2\nu}}-\log N} \ll \nu Q^{1/2\nu} \ll \log \log N,
\]
where we used the assumption that $k\leq 5\log N$ in the last step, and the implied constants are all absolute. 

Thus, we conclude that either $|A|\ll \log \log N$ or $|B|\ll (\log N)^{L+5}$, as desired.
\end{proof}

\section{Bounds on Diophantine tuples with property $D_{\leq d}(n)$}\label{sec:main}

In this section, we combine the results from all previous sections to prove Theorems~\ref{thm:Vd} and~\ref{thm:infty}.

\subsection{Proof of Theorem~\ref{thm:Vd}}\

Let $N=2|n|^{17}$ and let $A$ be a Diophantine tuple with property $D_{\leq d}(n)$. In view of Proposition~\ref{prop:large}, to prove the theorem, we can additionally assume that $N$ is sufficiently large and $A \subseteq \{1,2,\ldots, N\}$.  In particular, since $\sum_{p\mid n} \log p/\sqrt{p} \ll (\log |n|)^{1/2}$ \cite[Lemma 2.8]{KYY}, we can choose $N$ large enough so that
\begin{equation}\label{eq:bad}
\sum_{p\mid n} \frac{\log p}{\sqrt{p}} \leq \log N.    
\end{equation}

List the primes up to $d$ by $p_1,p_2, \ldots, p_m$ and set $r$ to be the product of these primes. By the prime number theorem, $\log r\ll d$. Let $c$ and $L$ be the positive constants from Corollary~\ref{cor:Linnik}. Without loss of generality, we may assume that $L \geq 10$ and $c\leq 1$. Let 
$$Q=r^L(c^{-1}D\phi(r)\sqrt{r}\log N)^2, \quad \mathcal{P}=\{p\leq Q: p \equiv 1 \pmod r, p \nmid n\}.
$$

For each prime $p \in \mathcal{P}$, let $A_p$ be the image of $A$ modulo $p$ and view $A_p$ as a subset of $\F_p$. Let $a, b \in A$ with $a \neq b$; then $ab+n \in V_d$, that is, there are $x\in \N$ and $1\leq j \leq m$, such that $ab+n=x^{p_j}$. Since $p \equiv 1 \pmod {p_i}$ for each $1\leq i \leq m$, it follows that $ab+n$, viewed as an element in $\F_p$, lies in the set $\bigcup_{i=1}^m\{y^{p_i}:y \in \F_p\}$. Thus, Proposition~\ref{prop:ap} implies that $|A_p|\leq D\sqrt{p}$, where $$D=(2^m+2)\prod_{i=1}^m (1-1/p_i)^{-1}\ll 4^m.$$ Thus, $\log D\ll m \ll d$.
By the prime number theorem, 
\begin{equation}\label{eq:numerator}
\sum_{p \in \mathcal{P}}\log p\leq \sum_{p \le Q} \log p \ll Q \ll r^{L+3}D^2 (\log N)^2.
\end{equation}
Since $Q\geq r^L$, Corollary~\ref{cor:Linnik} implies that
\begin{equation}\label{eq:L}
\sum_{\substack{p \equiv 1 \pmod r\\p \le Q}} \frac{\log p}{\sqrt{p}} \geq \frac{1}{\sqrt{Q}} \sum_{\substack{p \equiv 1 \pmod r\\p \le Q}} \log p\geq \frac{c\sqrt{Q}}{\phi(r)\sqrt{r}}\geq r^{L/2}D \log N.
\end{equation}
Thus, inequalities~\eqref{eq:bad} and~\eqref{eq:L} imply that
\begin{equation}\label{eq:denominator}
\sum_{p \in \mathcal{P}} \frac{\log p}{|A_p|} \geq \frac{1}{D} \sum_{\substack{p \equiv 1 \pmod r\\p \le Q}} \frac{\log p}{\sqrt{p}}-\sum_{p\mid n} \frac{\log p}{\sqrt{p}}\geq Lr^{L/2}\log N-\log N \gg r^{L/2}\log N.
\end{equation}
Applying Gallagher's sieve (Lemma~\ref{lem:GS}), inequalities~\eqref{eq:numerator} and~\eqref{eq:denominator}, and the fact $\log r \ll d$ and $\log D \ll d$, we conclude that
\[
|A|\leq \frac{\underset{p\in \mathcal{P}}\sum\log p - \log N}{\underset{p \in \mathcal{P}}\sum\frac{\log p}{|A_p|}-\log N}\ll r^{3+L/2} D^2\log N \ll e^{dL'} \log N
\]
for some absolute constant $L'$, where the implied constant is absolute, as required.


\subsection{Proof of Theorem~\ref{thm:infty}}\

By Proposition~\ref{prop:largeoo}, we may assume that $A \subseteq [1, M]$, where $M=(4|n|)^{17}$. We may assume that $M$ is sufficiently large. Let $G$ be the complete graph with vertex set $A$. For each $a b\in A$ with $a\neq b$, by definition, $ab+n$ is a perfect power; we can thus write $ab+n=x^p$ for some positive integer $x$ and a prime $p$, and we color the edge $ab$ by the smallest such $p$. Note that each prime $p$ we used to color some edge satisfies that $$p\leq \log_2(ab+n)\leq \log_2 (M^2+n)\leq 5\log M.$$ Let $L_1, L_2$ be the two absolute constants from \cref{thm:loglogN}. Then it follows from \cref{thm:loglogN} that $G$ does not contain a monochromatic $K_{\lceil L_1 \log \log M \rceil, \lfloor (\log M)^{L_2}\rfloor+1}$ as a subgraph. Thus, for each prime $p\leq 5\log M$, Lemma~\ref{lem:KST} implies that the number of edges in $G$ with color $p$ is at most 
$$
(\log M)^{L_2/(L_1 \log \log M) }
|A|^{2-1/\lceil L_1 \log \log M \rceil} + (L_1 \log \log M) |A|.
$$
Since the total number of edges in $G$ is $\binom{|A|}{2}$, it follows that
$$
\binom{|A|}{2} \leq 5\log M \bigg((\log M)^{L_2/(L_1 \log \log M) }
|A|^{2-1/\lceil L_1 \log \log M \rceil} + (L_1 \log \log M) |A|\bigg),
$$
and thus 
$$
|A|\leq (10e^{L_2/L_1} \log M)^{\lceil L_1 \log \log M \rceil}+2L_1 \log \log M\ll \exp(2L_1(\log \log |n|)^2),
$$
where the implied constant is absolute, as required.

\section{Proofs of conditional results}\label{sec:cond}

\subsection{Conditional estimates on $M_k(n)$: Proof of Theorem~\ref{thm:absolute}}\

Assume that $u<v$ are two elements in $A$. For each $b \in B$, there exist positive integers $x$ and $y$ such that $ub+n=x^k$ and $vb+n=y^k$; it follows that $uy^k-vx^k=(u-v)n$. In particular, for two distinct $b,b'\in b$, they induce a nontrivial integer solution $(x,y,z,w)$ to the Diophantine equation $$vx^k-uy^k-vz^k+uw^k=0$$ in the sense that $(x,y)\neq (z,w)$.

Suppose $|B|\geq 8$, then we can pick $8$ elements of $B$ to generate 
\begin{align*}
&vx_1^k-uy_1^k-vz_1^k+uw_1^k=0\\
&vx_2^k-uy_2^k-vz_2^k+uw_2^k=0\\
&vx_3^k-uy_3^k-vz_3^k+uw_3^k=0\\
&vx_4^k-uy_4^k-vz_4^k+uw_4^k=0
\end{align*}
More precisely, take $8$ distinct elements $b_1,b_2,b_3,b_4,c_1,c_2,c_3,c_4$ from $B$, and for each $1\leq i \leq 4$, define $x_i,y_i,z_i,w_i$ via the following rule:
\begin{equation}\label{eq:xyzw}    
ub_i+n=x_i^k, \quad vb_i+n=y_i^k, \quad uc_i+n=z_i^k, \quad vc_i+n=w_i^k.
\end{equation}
It follows that the matrix
\[
M:=M(b_1,b_2,b_3,b_4,c_1,c_2,c_3,c_4)=
\begin{pmatrix}
x_1^k & y_1^k & z_1^k & w_1^k \\
x_2^k & y_2^k & z_2^k & w_2^k \\
x_3^k & y_3^k & z_3^k & w_3^k \\
x_4^k & y_4^k & z_4^k & w_4^k 
\end{pmatrix}
\]
has determinant $0$ since $M (v, -u, -v, u)^T=\textbf{0}$. Thus, we have by the Leibniz formula for the determinant of $M$ that
\begin{equation}\label{eq:24sum}
0=\sum_{\sigma} (-1)^{\operatorname{sgn}(\sigma)} (x_{\sigma(1)}y_{\sigma(2)}z_{\sigma(3)}w_{\sigma(4)})^k,
\end{equation}
where $\sigma$ runs over $S_4$, the set of all $24$ permutations of $1,2,3,4$. Since $k\geq 25$, if we can show that the 24 numbers
$x_{\sigma(1)}y_{\sigma(2)}z_{\sigma(3)}w_{\sigma(4)}$
are pairwise distinct for $\sigma \in S_4$, then equation~\eqref{eq:24sum} will violate Conjecture~\ref{conj:LPS2}. 

Next assume that $|B|\geq 21737$. It suffices to show that we can always choose $8$ distinct elements $b_1,b_2,b_3,b_4,c_1,c_2,c_3,c_4$ from $B$, such that the 24 numbers
$x_{\sigma(1)}y_{\sigma(2)}z_{\sigma(3)}w_{\sigma(4)}$
are pairwise distinct for $\sigma \in S_4$. Here   $x_i,y_i,z_i,w_i$ are defined via equation~\eqref{eq:xyzw}. Equivalently, we study the matrix 
\[
N:=N(b_1,b_2,b_3,b_4,c_1,c_2,c_3,c_4)=
\begin{pmatrix}
x_1 & y_1 & z_1 & w_1 \\
x_2 & y_2 & z_2 & w_2 \\
x_3 & y_3 & z_3 & w_3 \\
x_4 & y_4 & z_4 & w_4 
\end{pmatrix}.
\]
For convenience, we introduce the following terminology. For each $1\leq i \leq 4$, we call a positive integer an \emph{$i \times i$ factor of $N$} if it appears as a term (up to the sign) in the Leibniz formula for the determinant of some $i \times i$ submatrix of $N$. Thus, our goal is to guarantee that all the $4 \times 4$ factors of $N$ are distinct, and we describe an algorithm to achieve this purpose. 

The key of the algorithm will be based on the following simple observation: the function $g(t)=\frac{ut+n}{vt+n}$ is monotone in $t$ whenever $vt+n>0$. Since
$$
\frac{ub_i+n}{vb_i+n}=\frac{x_i^k}{y_i^k},
$$
it follows that if $\frac{x_i}{y_i}$ is fixed, then $b_i$ is also uniquely fixed. Similarly, if $\frac{z_i}{w_i}$ is fixed, then $c_i$ is also uniquely fixed. In particular, as long as $b_1,b_2,b_3,b_4$ are distinct, we automatically have $x_iy_j\neq x_jy_i$ for each $1\leq i<j\leq 4$. Similarly, as long as $c_1,c_2,c_3,c_4$ are distinct, we automatically have $z_iw_j\neq z_jw_i$ for each $1\leq i<j\leq 4$.

We choose $b_1,b_2,b_3,b_4,c_1,c_2,c_3,c_4$ from $B$ step by step such that they are in strictly increasing order, such that at each step, for each $i \in \{1,2,3,4\}$, all the $i \times i$ factors in the matrix $N$ (with entries already determined) are distinct. We use the following algorithm to choose the elements $b_1, b_2, b_3, b_4$:
\begin{enumerate}
    \item choose $b_1$ to be the smallest element from $B$.
     \item choose the smallest element $b_2\in B$ such that $b_2>b_1$ and $x_2\neq y_1$.
    \item choose the smallest element $b_3\in B$ such that $b_3>b_2$ and $x_3\notin \{y_1,y_2\}$.
    \item choose the smallest element $b_4\in B$ such that $b_4>b_3$ and $x_4\notin \{y_1,y_2, y_3\}$.
\end{enumerate}

Since $|B|\geq 21737$, obviously the above algorithm went through. Note that for $i \in \{1,2\}$, all the $i\times i$ factors in the matrix $N$ with entries already determined are distinct. 

Next we pick $c_1, c_2,c_3,c_4$ step by step. For each $1\leq j \leq 4$, we pick $c_j$ using the following algorithm:
\begin{enumerate}
    \item Let $S_j$ be the collection of all possible factors in the matrix $N$ (with entries already determined).
    \item Let $S_j'=S_j \cup \{1\}$ and let $Q_j=S_j'/S_j'$ be the quotient set of $S'$.
    \item Pick $c_j\in B$ such that $z_j,w_j,w_j/z_j \notin Q_j$.
\end{enumerate}
We first note that $|S_j|$ is bounded from above by the total number of factors, so $$|S_j|\leq 16+\frac{16 \cdot 9}{2!}+\frac{16 \cdot 9 \cdot 4}{3!}+24=208.$$
It follows from that $|Q_j \cap (1, \infty)|\leq \binom{209}{2}=21736$. Note that for all $c_j\in B$, we have $z_j, w_j, \frac{w_j}{z_j}>1$. It follows that 
the above algorithm works since $|B|\geq 21737$.

Next, we claim that the choice of $c_j$ from the above algorithm guarantees that for each $i \in \{1,2,3,4\}$, all the $i \times i$ factors in the matrix $N$ (with entries already determined) are distinct. Let $i$ be fixed, consider two $i \times i$ factors $m_1$ and $m_2$ coming from different entries. Let $E$ be the set of labels of the entries building $m_1$ and $m_2$. We consider the following four cases:
\begin{enumerate}
\item Both $z_j$ and $w_j$ are in $E$. Without the loss of generality, we may assume that $m_1$ contains $z_j$ as an entry, and $m_2$ contains $w_j$ as an entry. Then if $m_1=m_2$, we have $m_1/z_j\in S_j'$ and $m_2/w_j=m_1/w_j \in S_j'$, and thus $w_j/z_j \in Q_j$, violating the assumption that $w_j/z_j \notin Q_j$.
\item $z_j$ is in $E$, and $w_j$ is not. Without the loss of generality, we may assume that $m_1$ contains $z_j$ as an entry. If $m_2$ also contains $z_j$ as an entry, then before we pick $c_j$, we can already guarantee $m_1/z_j\neq m_2/z_j$; if $m_2$ does not contain $z_j$ as an entry, then $m_1=m_2$ would imply that $m_1/z_j, m_1\in S_j'$ and thus $z_j \in Q_j$, violating the assumption. 

\item $w_j$ is in $E$, and $z_j$ is not. Similarly to the previous case, we have $m_1\neq m_2$.

\item None of $w_j$ and $z_j$ are in $E$. Then, before we pick $c_j$, we can already guarantee $m_1\neq m_2$.
\end{enumerate}
We conclude that $m_1\neq m_2$; this proves the claim.

Since $|B|\geq 21737$, we can run the above algorithm to construct the desired $c_1,c_2,c_3,c_4$ step by step, as required.

\subsection{Conditional estimates on $M_{\leq d}(n)$ and $\widetilde{f}(x)$: Proof of Theorem~\ref{thm:absolutedf} and Theorem~\ref{thm:LPS}}

\begin{proof}[Proof of Theorem~\ref{thm:absolutedf}]
Let $2\leq d<\infty$ and $n$ be a nonzero integer. Let $A$ be a Diophantine tuple with property $D_{\leq d}(n)$. Then $ab+n \in V_d$ for all $a,b \in A$ with $a\neq b$. Let $G$ be the complete graph with the vertex set $A$; for each edge $ab$ in $G$, we color it with the smallest prime $p$ such that $ab+n$ is a $p$-th power. By definition, all the primes we used to color the edges of $G$ are at most $d$. Let $C$ be the constant from \cref{cor:absolute}; then $G$ does not contain a monochromatic $K_{C+1}$ as a subgraph. Let $R$ be the Ramsey number $R(C+1,C+1,\cdots, C+1)$, where the number of $(C+1)$'s is $9$. 
Since there are $9$ primes less than $25$, it follows that $G$ contains no complete subgraph  $K_{R+1}$ such that all edges have color in $\{2,3,5,7,11,13,17,19,23\}$. Thus, Lemma~\ref{lem:Turan} implies that the number of edges in $G$ with color in $\{2,3,5,7,11,13,17,19,23\}$ is at most $\frac{1}{2}(1-\frac{1}{R-1})|A|^2$. For each prime $p$ such that $25<p\leq d$, \cref{thm:absolute} implies that $G$ does not contain a monochromatic $K_{2, 21737}$ in color $p$ as a subgraph and thus the number of edges in $G$ with color $p$ is $\ll |A|^{3/2}$ by Lemma~\ref{lem:KST}. We conclude that
$$
\binom{|A|}{2}-\frac{1}{2}(1-\frac{1}{R-1})|A|^2 \ll \pi(d) |A|^{3/2}
$$
and thus $|A|\ll \pi(d)^2\ll (d/\log d)^2$. This finishes the proof that $M_{\leq d}(n)\ll (d/\log d)^2$, where the implied constant is absolute.

Next, we give an upper bound on $\widetilde{f}(x)$. Let $A \subset [1,x] \cap \N$ such that there is an integer $n$ with $1\leq |n|\leq x$ such that $ab+n$ is a perfect power for all $a,b\in A$ with $a\neq b$; since $ab+n\leq x^2+x$ and $\log_2(x^2+x)\leq 5\log x$, it is necessary for $ab+n$ to be a $d$-th power for some $d\leq 5\log x$. This shows that $A$ is a Diophantine tuple with property $D_{\leq \lfloor 5 \log x \rfloor} (n)$ and thus $$|A|\leq M_{\leq \lfloor 5 \log x \rfloor} (n) \ll (\log x/\log \log x)^2,$$
as required.
\end{proof}

\begin{proof}[Proof of Theorem~\ref{thm:LPS}]
Assume that $x$ is sufficiently large. Pick a set $A \subseteq [1,x] \cap \N$ with $|A|=\widetilde{f}(x)$, such that there is some $1\leq |n| \leq x$ such that $ab+n$ is a perfect power for all $a,b \in A$ with $a\neq b$. Let $G$ be the graph with the vertex set $A$, such that there is an edge between two distinct vertices $a$ and $b$ if and only if $ab+n \in V_{\infty}$. By definition, $G$ is a complete graph. For each $a,b\in A$ such that $a\neq b$, we color the edge $ab$ with $1$ if $ab+n\in V_{25}$, otherwise we color the edge $ab$ with the smallest prime $p$ such that $ab+n$ is a $p$-th power. Note that each prime $p$ we used to color some edge satisfies that $25<p\leq \log_2 (x^2+n)\leq 5\log x$. By Theorem~\ref{thm:Vd}, there is an absolute constant $C$, such that $G$ does not contain a monochromatic $K_{\lfloor C\log x\rfloor+1}$ with color $1$ as a subgraph. Lemma~\ref{lem:Turan} then implies that the number of edges in $G$ with color $1$ is at most $\frac{1}{2}(1-\frac{1}{\lfloor C\log x\rfloor})|A|^2$. This implies that the number of edges in $G$ with color being a prime is at least $$\binom{|A|}{2}-\frac{1}{2}(1-\frac{1}{\lfloor C\log x\rfloor})|A|^2\geq \frac{|A|^2}{\lfloor C\log x\rfloor}-|A| \gg \frac{|A|^2}{\log x}.$$
For each prime $p$ with $25<p<5\log x$, by Theorem~\ref{thm:absolute}, $G$ does not contain a monochromatic $K_{2,21737}$ with color $p$ as a subgraph; thus, Lemma~\ref{lem:KST} implies that the number of edges in $G$ with color $p$ is $\ll |A|^{3/2}$. It follows that $\frac{|A|^2}{\log x} \ll \log x \cdot |A|^{3/2}$ and thus $\widetilde{f}(x)=|A|\ll (\log x)^4$.
\end{proof}

\subsection{Conditional estimates on $M_{\leq \infty}(n)$: Proof of Theorem~\ref{thm:ABC}}

\begin{proof}[Proof of Theorem~\ref{thm:ABC}]
Let $A$ be a Diophantine tuple with property $M_{\leq \infty}(n)$.
Under the ABC conjecture, by \cite[Lemma 1]{BDHL11}, there is an absolute constant $c_0$, such that there do not exist $5$ distinct positive integers $a_1,a_2,a_3,a_4,a_5$, each at least $c_0|n|^3$, with the property that for each $1\leq i<j \leq 3205$, $a_ia_j+n={x_{ij}}^{k_{ij}}$ for some integers $x_{ij}$ and $k_{ij}$ with $k_{ij}\geq 3205$. Let $$N=\max (c_0|n|^3, 2|n|^{17}).$$ Define $A_1=A \cap [1,N]$ and $A_2=A \cap (N, \infty]$. By definition, $|A_1|\ll \widetilde{f}(N)\ll \widetilde{f}(2|n|^{17})$. It suffices to show $|A_2|$ is absolutely bounded, so that we have $$|A|=|A_1|+|A_2|\ll \widetilde{f}(2|n|^{17}) \ll \exp(L (\log \log |n|)^2)$$ for some absolute constant $L$ from Corollary~\ref{cor:f}.

Next, we show that $|A_2|$ is absolutely bounded. Label the $452$ primes less than $3205$ by $2=p_1<p_2<\ldots<p_{452}=3203$. We build a complete graph $G$ with vertex set $A_2$, and we color the edge connecting two distinct elements $a,a'\in A_2$ with the color $i$ by the following rule:
\begin{itemize}
    \item If $aa'+n$ is a perfect $k$-th power for some $k\geq 3205$, then $i=0$;
    \item Otherwise, set $i$ to be the smallest number such that $aa'+n$ is a perfect $p_i$-th power.
\end{itemize}
In view of the above discussion, $G$ contains no monochromatic complete subgraph $K_5$ in color $0$. By Lemma~\ref{lem:D_2}, $G$ contains no monochromatic complete subgraph $K_{22}$ in color $1$. Let $C$ be the Ramsey number $R(5,22)$; then by
definition, $G$ contains no complete subgraph $K_{C}$ with all edges in color $0$ or $1$. Lemma~\ref{lem:Turan} then implies that the number of edges in $G$ with color $0$ or $1$ is at most $\frac{1}{2} (1-\frac{1}{C-1})|A_2|^2$. By Corollary~\ref{cor:contribution}, the number of edges in $G$ with color $2$ is at most $8|A_2|^{5/3}$, and for each $3\leq i \leq 452$, the number of edges in $G$ with color $i$ is at most $7|A_2|^{3/2}$. It follows that 
$$
\frac{|A_2|(|A_2|-1)}{2} \leq \frac{1}{2} \bigg(1-\frac{1}{C-1}\bigg)|A_2|^2+8|A_2|^{5/3}+450 \cdot 7|A_2|^{3/2},
$$
and thus
$$
|A_2|\leq (C-1)+ 16(C-1)|A_2|^{2/3}+ 6300 |A_2|^{1/2}.
$$
We conclude that $|A_2|$ is absolutely bounded, as required.
\end{proof}

\section*{Acknowledgments}
The second author thanks Andrej Dujella, Greg Martin, J\'ozsef Solymosi, and Zixiang Xu for helpful discussions at the early stages of the project. The authors thank the anonymous referees for their valuable comments and suggestions. The research of the second author was supported in part by an NSERC fellowship.

\bibliographystyle{abbrv}
\bibliography{main}

\end{document}